\documentclass[12pt,reqno]{amsart}
\usepackage{mathrsfs}
\usepackage{amsfonts}

\usepackage{amssymb,amsmath,graphicx,amsfonts,euscript}
\usepackage{color}

\newtheorem{thm}{Theorem}[section]

\newtheorem{Pros}[thm]{Proposition}

\newtheorem{rem}[thm]{Remark}

\newtheorem{lemma}[thm]{Lemma}

\numberwithin{equation}{section} \allowdisplaybreaks \voffset=-0.2in

\begin{document}
\bigskip\bigskip

\centerline{\Large\bf  Global regularity
for the 2D Oldroyd-B model   }

\smallskip

\centerline{\Large\bf in the corotational case }

\bigskip

\centerline{Zhuan Ye, Xiaojing Xu}
\medskip

\centerline{School of Mathematical Sciences, Beijing Normal
University,}
\medskip

\centerline{Laboratory of Mathematics and Complex Systems, Ministry
of Education,}
\medskip

\centerline{Beijing 100875, People's Republic of China}

\medskip

\centerline{E-mails: \texttt{yezhuan815@126.com; xjxu@bnu.edu.cn}}

\bigskip

{\bf Abstract:}~~%
This paper is dedicated to the Oldroyd-B model with fractional dissipation
$(-\Delta)^{\alpha}\tau$ for any $\alpha>0$. We establish the global smooth solutions to the Oldroyd-B model in the corotational case with arbitrarily small fractional powers of the Laplacian in two spatial dimensions. The methods described here are quite different from the tedious iterative approach used in recent paper \cite{XY}. Moreover, in the Appendix we provide some a priori estimates to the  Oldroyd-B model in the critical case which may be useful and of interest for future improvement. Finally, the global regularity to to the Oldroyd-B model in the corotational case with $-\Delta u$ replaced by $(-\Delta)^{\gamma}u$ for $\gamma>1$ are also collected in the Appendix.
Therefore our result is more closer to the resolution of the well-known global regularity issue on the critical 2D Oldroyd-B
model.

{\vskip 1mm
 {\bf AMS Subject Classification 2010:}\quad 76A10; 76D03;
76A05; 76N99.

 }

\vskip .4in
\section{Introduction}\label{intro}
The classical Oldroyd-B type model with diffusive stress can be written as
\begin{equation}\label{OLD1}
\left\{\aligned
&\partial_{t}u+(u\cdot\nabla)u-\nu\Delta u+\nabla \pi=\kappa \nabla\cdot \tau,\\
&\partial_{t}\tau+(u\cdot\nabla)\tau+\beta\tau-\mu\Delta \tau-Q(\nabla u,\tau)=\gamma\mathcal {D}u,\\
&\nabla\cdot u=0,\\
&u(x,0)=u_{0}(x),\,\tau(x,0)=\tau_{0}(x),
\endaligned \right.
\end{equation}
where $\nu\geq0,\,\mu\geq0,\,\beta\geq0,\,\kappa>0$ and $\gamma>0$
are real constant parameters. $u=u(x,t)=(u_{1}(x,t),u_{2}(x,t))\in
\mathbb{R}^{2}$ denote the velocity of the fluid, $\pi=\pi(x,t)\in
\mathbb{R}$ denotes scalar pressure and
$\tau=\tau(x,t)$ is the non-Newtonian part of the stress tensor,
($\tau(x,t)$ is a $(2,2)$ symmetric matrix). $\mathcal {D}u$ is the
symmetric part of the velocity gradient, namely $\mathcal
{D}u=\frac{1}{2}(\nabla u+\nabla u^{\top})$. $Q$ is a given bilinear
form
\begin{eqnarray}\label{t112}Q(\nabla u,\tau)=\Omega\tau-\tau\Omega+b(\mathcal {D}u\tau+\tau\mathcal
{D}u),\end{eqnarray} where $\Omega=\frac{1}{2}(\nabla u-\nabla
u^{\top})$ is the skew symmetric part of $\nabla u$ and $b\in [-1,\,1]$
is a parameter.

Let us say some words about the derivation of the system
(\ref{OLD1}) with $\mu=0$ (Indeed, this is the classical case). Incompressible fluids with constant density are being described by the set of equations
 \begin{equation}\label{LLOLD}
\left\{\aligned
&\partial_{t}u+(u\cdot\nabla)u=\varrho\nabla\cdot \varphi,\\
&\nabla\cdot u=0,
\endaligned \right.
\end{equation}
 where $\varrho$ is a constant parameter and $\varphi$ is the stress tensor which can be decomposed into $\varphi=-\widetilde{\pi} {\rm \textbf{Id}}+\tau$. Here $\widetilde{\pi}$ denotes the pressure of the fluid and {\rm \textbf{Id}} is the identity tensor.\\
 Recall the Oldroyd-B \cite{O} constitutive law of differential type
\begin{eqnarray}\label{LL1}
\tau+\lambda_{1}\frac{\mathcal{D}_{b}\tau}{\mathcal{D}t}=2\sigma\big(\mathcal {D}u+\lambda_{2}\frac{\mathcal{D}_{b}\mathcal{D}u}{\mathcal{D}t}\big),
\end{eqnarray}
where $\frac{\mathcal{D}_{b}}{\mathcal{D}t}$ denotes the "objective derivative" as follows
$$\frac{\mathcal{D}_{b}\tau}{\mathcal{D}t}:=\partial_{t}\tau+(u\cdot\nabla)\tau-Q(\nabla u,\tau).$$
Here $\sigma$ is the total viscosity of the
fluid, $\lambda_{1},\,\lambda_{2}$ are the relaxation time and retardation time with $0<\lambda_{2}\leq\lambda_{1}$. The symmetric tensor of constrains $\tau$ could be decomposed into the Newtonian
part $\tau_{N}$ and the elastic part $\tau_{E}$, namely
\begin{eqnarray}\label{LL2}
\tau=\tau_{N}+\tau_{E},\quad {\rm with} \,\,\, \tau_{N}=2\widetilde{\sigma}\mathcal {D}u,\end{eqnarray}
where $\widetilde{\sigma}=\frac{\lambda_{2}\sigma}{\lambda_{1}}$ is the solvent viscosity.\\ Combining (\ref{LL1}) and (\ref{LL2}), it is not difficult to check that
 $\tau_{E}$ meets the following equality
$$\tau_{E}+\lambda_{1}\frac{\mathcal{D}_{b}\tau_{E}}{\mathcal{D}t}=
2(\sigma-\widetilde{\sigma})\mathcal {D}u,$$
where $\sigma-\widetilde{\sigma}$ is the polymer viscosity. \\
Setting with some abuse of notation $\tau_{E}$ by $\tau$, we deduce from (\ref{LLOLD}) and above inequality that
  \begin{equation}\nonumber
\left\{\aligned
&\partial_{t}u+(u\cdot\nabla)u-\varrho\widetilde{\sigma}\Delta u+\nabla (\varrho\widetilde{ \pi})=\varrho\nabla\cdot \tau,\\
&\tau+\lambda_{1}\big(\partial_{t}\tau+(u\cdot\nabla)\tau-Q(\nabla u,\tau)\big)=
2(\sigma-\widetilde{\sigma})\mathcal {D}u.
\endaligned \right.
\end{equation}
 Thus, we can immediately obtain the system
(\ref{OLD1}) with $\mu=0$ by taking
$$\nu=\frac{\lambda_{2}\sigma\varrho}{\lambda_{1}},\,\,\,\kappa=\varrho,
\,\,\,\beta=\frac{1}{\lambda_{1}},\,\,\, \gamma=\frac{2\sigma}{\lambda_{1}}\big(1-\frac{\lambda_{2}}{\lambda_{1}}\big) \quad \mbox{and}\,\,\, \pi=\varrho\widetilde{ \pi}.$$
The above system of equations originally was introduced by Oldroyd
\cite{O} which is one of the basic macroscopic models for
visco-elastic flows such as polymer flows; fluids of this type have both elastic
properties and viscous properties. We refer the readers to
\cite{O,CM,FGO} for more discussions and the derivation of Oldroyd-B
model (\ref{OLD1}).

Due to their physical applications and mathematical significance, the Oldroyd-B model have recently attracted considerable attention and many important results on the existence theory and regularity criterion have been established.
In the most interesting case $\nu>0$ and $\mu=0$ (which is the classical case),
existence of local strong solutions to the
Oldroyd-B model was proved by Guillop$\rm\acute{ e}$ and Saut in
\cite{GS,GS1}. In the frame of critical Besov spaces, Chemin and
Masmoudi constructed global solutions to the incompressible Oldroyd-B model with small
initial data (see also Chen and Miao \cite{CM1}). In addition, non-blow
up criteria for Oldroyd-B model were given in \cite{CM,LMZ}.
For the Oldroyd-B fluids with diffusive stress (more precisely, the system (\ref{OLD1}) with $\nu>0,\,\mu>0$ and $b=1$) Constantin and Kliegel \cite{CK} established the existence and uniqueness of global strong solutions in the two dimensional case. The proof is based on the energy method and the use of the maximum principle. Very recently, Elgindi and Rousset \cite{EK} proved the global existence
of strong solutions with large data to the system (\ref{OLD1}) in
the case $\nu=0,\,\mu>0$ and $Q=0$. Moreover, they also obtained the
global well-posedness in the case of $\nu=0,\,\mu>0$ and $Q$ given
by (\ref{t112}) with small initial data. Many works have been
devoted to obtaining the global well-posedness in the case of small
initial data (see, e.g., \cite{LLZ,LZ,Zi,ZFZ}). Many other interesting results on the
Oldroyd-B and related models have been established (see, e.g.,
\cite{FHZ,FGO,HL,HW,HW1,L,LN,L1,LLZ0,M1,M2,M3} and the
references therein).

We also mention the following system in corotational case which is
an immediate case of the system (\ref{OLD1}) with $\nu>0$, $\mu=0$
and $b=0$:
\begin{equation}\label{New}
\left\{\aligned
&\partial_{t}u+(u\cdot\nabla)u-\nu\Delta u+\nabla \pi=\kappa \nabla\cdot \tau,\\
&\partial_{t}\tau+(u\cdot\nabla)\tau+\beta\tau+\eta(\tau\Omega-\Omega\tau)=\gamma\mathcal {D}u,\\
&\nabla\cdot u=0,\\
&u(x,0)=u_{0}(x),\,\tau(x,0)=\tau_{0}(x).
\endaligned \right.
\end{equation}
 The global existence of weak solutions (without uniqueness) to above system (\ref{New}) was proved
by Lions and Masmoudi \cite{LN}. Bejaoui and Mohamed Majdoub
\cite{BM} generalized the results in \cite{LN}. However, the global
existence of smooth solutions is open and quite challenging (see
\cite{EK} for more details). We mention that global weak for above system (\ref{New})
with $\eta(\tau\Omega-\Omega\tau)$ replaced by general $Q$, namely (\ref{t112}),
is still open up to now. As pointed out in \cite{EK}, in the case
where $\eta=0$ and $\nu>0$, the global existence of
smooth solutions to above system (\ref{New}) is also open and quite challenging.
Since there is no global existence
result for general initial data, we would like to add the fractional
dissipation $(-\Delta)^{\alpha}\tau$ to the stress tensor $\tau$
equation to guarantee the global well-posedness result.
 Therefore, it is natural to consider
the following Oldroyd-B model in the corotational case (with $b=0$)
with fractional dissipation
\begin{equation}\label{OLD}
\left\{\aligned
&\partial_{t}u+(u\cdot\nabla)u-\nu\Delta u+\nabla \pi=\kappa \nabla\cdot \tau,\\
&\partial_{t}\tau+(u\cdot\nabla)\tau+\beta\tau+\eta(\tau\Omega-\Omega\tau)+\mu(-\Delta)^{\alpha} \tau=\gamma\mathcal {D}u,\\
&\nabla\cdot u=0,\\
&u(x,0)=u_{0}(x),\,\tau(x,0)=\tau_{0}(x),
\endaligned \right.
\end{equation}
where $\alpha\in [0,\,1]$ (we mention that the small $\alpha$ is the main focus of this paper) and the
fractional Laplacian operator $(-\Delta)^{\alpha}$ is defined
through the Fourier transform, namely
$\widehat{(-\Delta)^{\alpha} f}(\xi)=|\xi|^{2\alpha}\widehat{f}(\xi).$
We make the convention that by $\alpha=0$ we mean
that there is no dissipation in the second equation of (\ref{OLD}).

Let us compare the Oldroyd-B model (\ref{OLD}) with the
two-dimensional Boussinesq system with partial dissipation which has
been considered by lots of works, just name a few (see, e.g.,
\cite{CW1,CHAE,HK0,HK1,HK2,HL1,JMWZ}). The coupling in the
Boussinesq system is simpler than the one in the Oldroyd models
since the vorticity equation is forced by the gradient of the
temperature in terms of the Boussinesq system, but that the
temperature solves an unforced convection-diffusion equation.
However, for the Oldroyd-B model (\ref{OLD}), the second equation
 has a forced term $\mathcal {D}u$ which prevents us from obtaining
 the $L^{q}-$norm of $\tau$ for $2\leq q\leq \infty$ while we can immediately get any  $L^{q}-$norm
 of temperature for the Boussinesq system. This is the big
 difficulty since the Oldroyd-B model is fully coupled. This is also
 the reason that we can not directly and fully follow the methods introduced by
 Hmidi, Keraani and Rousset \cite{HK1}.

The above system (\ref{OLD}) with $\mu=\eta=0$ is still a challenging
open problem as discussed in \cite{EK}. In this paper we would like
to show that for any small power $\alpha>0$ the system (\ref{OLD}) always admits a unique global smooth solution.

For the sake of simplicity,  we will limit ourselves to
$\nu=\mu=\eta=\kappa=\gamma=1$ and $\beta=0$ in the rest of the paper. The main result of this paper is the following
\begin{thm}\label{Th1} Suppose that $\alpha>0$ and $(u_{0}, \tau_{0})\in H^{s}(\mathbb{R}^{2})\times
H^{s}(\mathbb{R}^{2})$ for $s>2$. Then there exists a unique global smooth solution pair
$(u(x,t),\,\tau(x,t))$ to the system (\ref{OLD}) such that for any given
$T>0$
$$u\in C([0, T]; H^{s}(\mathbb{R}^{2}))\cap L^{2}([0, T]; H^{s+1}(\mathbb{R}^{2})),$$
$$\tau\in C([0, T]; H^{s}(\mathbb{R}^{2}))\cap L^{2}([0, T];
H^{s+\alpha}(\mathbb{R}^{2})).$$
\end{thm}
\begin{rem}\rm
Indeed, the case $\nu>0,\,\mu>0,\,\alpha=1$ the system (\ref{OLD})
can be considered as a subcritical case in dimension two. The basic
energy method and the use of the maximum principle is enough to
ensure the existence of global smooth solutions (see for example
\cite{CK}). Therefore, in this paper we only focus upon the power
$0<\alpha<1$ (in fact, the number $\alpha$ is arbitrarily small). To the best of our knowledge, it has not been studied.
As a matter of fact, for the system (\ref{OLD}) with $ \alpha<1$, it
seems impossible to get global smooth solutions by using the direct
energy estimates. However, at present we are not able to show the global regularity result for the system (\ref{OLD}) in the case $\alpha=0$, even if $\eta=0$, but for this case some a priori estimates will be provided In the Appendix which may be useful and of interest for future improvement. Thus it is still an extremely complicated and interesting problem for the system (\ref{OLD}) with only $\alpha=0$.
\end{rem}
\begin{rem}\rm
It is worthy to emphasize that in recent paper \cite{XY}, we also
established the global regularity to the system (\ref{OLD}) with
$\eta=0$ and $\alpha>0$ by using the the iterative approach. We
point out that the shortcoming in \cite{XY} is the tedious
computations. At this point we note that also our approach for
obtaining the global solution to (\ref{OLD}) seems to be quite
different from the iterative approach in \cite{XY} and can be
applied to (\ref{OLD}) with $\eta=0$.
\end{rem}

\begin{rem}\rm
Note that Theorem \ref{Th1} is a global existence result of smooth
solutions without any size restriction on the initial data.
\end{rem}

Let us remark that considering the system (\ref{OLD}) with $\mu=0$ and $-\Delta u$ replaced by $(-\Delta)^{\gamma} u$ for any $\gamma>1$, namely
\begin{equation}\label{FROLD}
\left\{\aligned
&\partial_{t}u+(u\cdot\nabla)u+(-\Delta)^{\gamma} u+\nabla \pi= \nabla\cdot \tau,\\
&\partial_{t}\tau+(u\cdot\nabla)\tau+(\tau\Omega-\Omega\tau)=\mathcal {D}u,\\
&\nabla\cdot u=0,\\
&u(x,0)=u_{0}(x),\,\tau(x,0)=\tau_{0}(x),
\endaligned \right.
\end{equation}
one can also show that the corresponding system admits a unique global smooth solution, more precisely
\begin{thm}\label{Th2} Suppose that for any $\gamma>1$ and $(u_{0}, \tau_{0})\in H^{s}(\mathbb{R}^{2})\times
H^{s}(\mathbb{R}^{2})$ for $s>2$. Then there exists a unique global smooth solution pair
$(u(x,t),\,\tau(x,t))$ to the system (\ref{FROLD}) such that for any given
$T>0$
$$u\in C([0, T]; H^{s}(\mathbb{R}^{2}))\cap L^{2}([0, T]; H^{s+\gamma}(\mathbb{R}^{2})),$$
$$\tau\in C([0, T]; H^{s}(\mathbb{R}^{2})).$$
\end{thm}
We want to point out that the proof of Theorem \ref{Th2} can be achieved by the similar argument applied in proving Theorem \ref{Th1}, and thus the details are given in the Appendix.

Finally, the rest of this paper is organized as follows. In Section
2, we give the proof of the main result, namely, Theorem \ref{Th1}.  In the Appendix we provide some a priori estimates to the system (\ref{OLD}) with $\alpha=\eta=0$ which may be useful and of interest for future improvement. Moreover, we also give the details of the proof of Theorem \ref{Th2} in the Appendix.

\vskip .4in
\section{proof of Theorem \ref{Th1}}\setcounter{equation}{0}
This section is devoted to the proof of Theorem \ref{Th1}.
 Before proving the theorem, we first introduce the following
conventions and notations which will be used throughout this paper.
Throughout this paper, the letter $C$ denotes a general constant which may
be different from line to line. We shall sometimes use
the natation $A\lesssim B$ which stands for $A\leq C B$.
Finally, we denote: $\omega={\rm curl}
(u)=\partial_{1}u_{2}-\partial_{2}u_{1}$, $\nabla\cdot u={\rm div} (u)$,
$\Lambda=(-\Delta)^{\frac{1}{2}}$ and the standard commutator
notation $[\mathcal {R},\,f]g=\mathcal {R}(fg)-f\mathcal {R}g$. For
three $n$ order matrices  $\mathbb{A}$, $\mathbb{B}$ and $\tau$, we denote
$\mathbb{A}: \mathbb{B}=\sum_{i,j=1}^{n}a_{ij}b_{ij}$ where $a_{ij}$
and $b_{ij}$ are the components of matrices  $\mathbb{A}$ and
$\mathbb{B}$, respectively and $(\nabla\cdot\tau)_{j}=\sum_{i=1}^{2}\partial_{i}\tau_{ij}$.

The existence and uniqueness of local smooth solutions can be done
without any difficulty as in the case of the Euler and Navier-Stokes
equations, thus it is sufficient to establish a priori estimates.

To prove the main result, we need to establish two key lemmas,
namely Lemmas \ref{L32} and \ref{L33}. First, we can easily derive
the following energy estimate from the system (\ref{OLD}) which
holds true for any $\alpha\geq0$.
\begin{lemma}\label{L31}
For any corresponding solution $(u, \tau)$ of (\ref{OLD}), there
exist some constants $C$ such that for any $T>0$
\begin{eqnarray}\label{t201}
\|u(t)\|_{L^{2}}^{2}+\|\tau(t)\|_{L^{2}}^{2}+2\int_{0}^{T}{(\|\nabla
u\|_{L^{2}}^{2}+\|\Lambda^{\alpha}\tau\|_{L^{2}}^{2})\,dt}=
\|u_{0}\|_{L^{2}}^{2}+\|\tau_{0}\|_{L^{2}}^{2}\leq C<\infty
\end{eqnarray}
for any $t\in[0, T].$
\end{lemma}
\begin{proof}[\textbf{Proof of  Lemma \ref{L31}}]
Taking the inner product of $(\ref{OLD})_{1}$ with $u$ and the inner
product of $(\ref{OLD})_{2}$ with $\tau$, using the divergence free
property and summing up them, we easily get
\begin{eqnarray}\label{t2010}
\frac{1}{2}\frac{d}{dt}(\|u(t)\|_{L^{2}}^{2}+\|\tau(t)\|_{L^{2}}^{2})+\|\nabla
u\|_{L^{2}}^{2}+\|\Lambda^{\alpha}\tau\|_{L^{2}}^{2} =0,
\end{eqnarray}
where the following the cancelation identities have been applied
$$\int_{\mathbb{R}^{2}}{(\nabla\cdot\tau)\cdot u
\,dx}+\int_{\mathbb{R}^{2}}{\mathcal {D}u:\tau\,dx}=0 \quad \mbox{and}
\quad  \int_{\mathbb{R}^{2}}{(\tau\Omega-\Omega\tau):\tau\,dx}=0,\quad \big({\rm see } (\ref{tttt11})\big).$$
Integrating (\ref{t2010}) from $0$ to $t$, we obtain the desired
result.
\end{proof}
The following lemma proves a global bound for $\|\tau(t)\|_{L^{r}}$ and $\int_{0}^{T}{\|\omega(t)\|_{L^{r}}^{\frac{2r}{r-2}}\,dt}$ for any $2<r<\infty$.
This global bound is valid for any $\alpha>0$ which will play a
significant role in obtaining the higher integrability in terms of
the vorticity $w$ and the  stress tensor $\tau$. More precisely, we have the following lemma
\begin{lemma}\label{L32}
Assume that $\alpha>0$ and any $2<r<\infty$. For any corresponding
solution $(u, \tau)$ of (\ref{OLD}), there exist some constants $C$
such that for any $T>0$
\begin{eqnarray}\label{t202}
\|\tau(t)\|_{L^{r}}\leq C<\infty,\quad
\int_{0}^{T}{\|\omega(t)\|_{L^{r}}^{\frac{2r}{r-2}}\,dt}\leq C<\infty
\end{eqnarray}
for any $t\in[0, T].$
\end{lemma}
\begin{proof}[\textbf{Proof of  Lemma \ref{L32}}]
In order to get the above estimate, we first apply operator ${\rm curl}$
to the equation $(\ref{OLD})_{1}$ to obtain the vorticity $w$
equation as follows
\begin{eqnarray}\label{t203}\partial_{t}\omega+(u\cdot\nabla)\omega-\Delta \omega={\rm curl \,div } (\tau).\end{eqnarray}
However, the "vortex stretching" term ${\rm curl \,div } (\tau)$ appears to
prevent us from proving any global bound for $\omega$, even though one
combines this vorticity equation with the equation
$(\ref{OLD})_{2}$. To overcome this difficulty, we exploit the method introduced by
Hmidi, Keraani and Rousset \cite{HK1,HK2} to treat the critical Boussinesq equations.
Their natural idea would be to eliminate ${\rm curl \,div }
(\tau)$ from the vorticity equation. Here we would like to mention that Elgindi-Rousset \cite{EK} first bring the trick to Oldroyd-B model.
To realize this idea, we take
$\mathcal {R}$ as the singular integral operator
$$\mathcal {R}=(-\Delta)^{-1}{\rm curl \,div }.$$
Applying the operator $\mathcal {R}$ to equation $(\ref{OLD})_{2}$, one has
\begin{eqnarray}\label{ZZtt}\partial_{t}\mathcal {R}\tau+(u\cdot\nabla)\mathcal {R}\tau=\mathcal {R}\mathcal {D}u-\mathcal {R}\Lambda^{2\alpha} \tau-[\mathcal {R},\,u\cdot\nabla]\tau-\mathcal {R}(\tau\Omega-\Omega\tau).\end{eqnarray}
We set the combined quantity $\Gamma$ as follows
$$\Gamma=\omega-\mathcal {R}\tau.$$ By combining (\ref{t203}) and ($\ref{ZZtt}$), it is easy to verify that $\Gamma$ satisfies
\begin{eqnarray}\label{t204}\partial_{t}\Gamma+(u\cdot\nabla)\Gamma-\Delta \Gamma
=\mathcal {R}\Lambda^{2\alpha}\tau-\mathcal {R}\mathcal
{D}u+[\mathcal {R},\,u\cdot\nabla]\tau+ \mathcal
{R}(\tau\Omega-\Omega\tau).\end{eqnarray} Testing (\ref{t204}) by $\Gamma$, we obtain, after integration by parts
\begin{eqnarray}\label{t205}
\frac{1}{2}\frac{d}{dt}\|\Gamma\|_{L^{2}}^{2}+\|\nabla\Gamma\|_{L^{2}}^{2}=\mathcal{N}_{1}
+\mathcal{N}_{2}+\mathcal{N}_{3},
\end{eqnarray}
where
$$\mathcal{N}_{1}=\int_{\mathbb{R}^{2}}{(\mathcal {R}\Lambda^{2\alpha}\tau-\mathcal {R}\mathcal
{D}u) \Gamma\,dx},\quad \mathcal{N}_{2}=\int_{\mathbb{R}^{2}}{[\mathcal
{R},\,u\cdot\nabla]\tau \Gamma\,dx},$$
$$\mathcal{N}_{3}=\int_{\mathbb{R}^{2}}{\mathcal
{R}(\tau\Omega-\Omega\tau)\Gamma\,dx}.$$
 In what follows, we will deal
with each term on the right-hand side of (\ref{t205}) separately.\\
Young inequality and interpolation inequality imply that
\begin{eqnarray}\label{t206}
\mathcal{N}_{1} &\leq& \|\Lambda^{\alpha}\tau\|_{L^{2}} \|\mathcal
{R}\Lambda^{\alpha}\Gamma \|_{L^{2}}+\|\mathcal {R}\mathcal
{D}u\|_{L^{2}}\|\Gamma
\|_{L^{2}}\nonumber\\
&\lesssim& \|\Lambda^{\alpha}\tau\|_{L^{2}} \|\Lambda^{\alpha}\Gamma
\|_{L^{2}}+\|\mathcal {D}u\|_{L^{2}}\|\Gamma
\|_{L^{2}}\nonumber\\
&\lesssim& \|\Lambda^{\alpha}\tau\|_{L^{2}} \|\Gamma
\|_{L^{2}}^{1-\alpha}\|\nabla\Gamma \|_{L^{2}}^{\alpha}+\|\nabla
u\|_{L^{2}}\|\Gamma \|_{L^{2}}\nonumber\\
&\leq& \frac{1}{4}\|\nabla\Gamma\|_{L^{2}}^{2}+C\|\Gamma
\|_{L^{2}}^{2}+C(\|\Lambda^{\alpha}\tau\|_{L^{2}}^{2} +\|\nabla
u\|_{L^{2}}^{2}),
\end{eqnarray}
where we have used the following facts: $\|\mathcal
{R}f\|_{L^{p}}\leq C\|f\|_{L^{p}}$ for any $p\in(1,\,\infty)$ and
the interpolation $\|\Lambda^{\alpha}\Gamma \|_{L^{2}}\leq C
\|\Gamma \|_{L^{2}}^{1-\alpha}\|\nabla\Gamma \|_{L^{2}}^{\alpha}$
for any $0\leq\alpha\leq1.$\\
The term $\mathcal{N}_{2}$ is much more involved. To estimate it appropriately, we apply the following
commutator estimate (see Theorem 3.3 in \cite{HK1})
\begin{eqnarray}\label{t207}\|[\mathcal {R},\,u]f\|_{H^{s}}\leq C(s)
(\|\nabla u\|_{L^{2}}\|f\|_{B_{\infty,2}^{s-1}}+\|
u\|_{L^{2}}\|f\|_{L^{2}}),\end{eqnarray} for any smooth
divergence-free vector field  and any $0<s<1$. Here $B_{p,r}^{s}$
with $s\in \mathbb{R}$ and $p,\,r\in [1, \infty]$ denotes an
inhomogeneous Besov space. \\It follows from the above commutator
estimate (\ref{t207}) that
\begin{eqnarray}\label{t208}
\mathcal{N}_{2} &=& \int_{\mathbb{R}^{2}}{\nabla\cdot[\mathcal
{R},\,u]\tau \Gamma\,dx}\quad (\nabla\cdot u=0)\nonumber\\
&\lesssim& \|[\mathcal {R},\,u]\tau\|_{\dot{H}^{\frac{r-2}{2r}}}
\|\Gamma
\|_{\dot{H}^{\frac{r+2}{2r}}}\nonumber\\
&\lesssim& \|[\mathcal {R},\,u]\tau\|_{{H}^{\frac{r-2}{2r}}}
\|\Gamma
\|_{\dot{H}^{\frac{r+2}{2r}}}\nonumber\\
&\lesssim& (\|\nabla
u\|_{L^{2}}\|\tau\|_{B_{\infty,2}^{\frac{-r-2}{2r}}}+\|
u\|_{L^{2}}\|\tau\|_{L^{2}}) \|\Gamma
\|_{L^{2}}^{\frac{r-2}{2r}}\|\nabla\Gamma
\|_{L^{2}}^{\frac{r+2}{2r}}\nonumber\\
&\lesssim& (\|\nabla u\|_{L^{2}}\|\tau\|_{L^{r}}+\|
u\|_{L^{2}}\|\tau\|_{L^{2}}) \|\Gamma
\|_{L^{2}}^{\frac{r-2}{2r}}\|\nabla\Gamma
\|_{L^{2}}^{\frac{r+2}{2r}}\nonumber\\
&\leq& \frac{1}{8}\|\nabla\Gamma\|_{L^{2}}^{2}+ C\|\nabla
u\|_{L^{2}}^{\frac{4r}{3r-2}}\|\tau\|_{L^{r}}^{\frac{4r}{3r-2}}
\|\Gamma\|_{L^{2}}^{\frac{2(r-2)}{3r-2}}+C(\|
u\|_{L^{2}}\|\tau\|_{L^{2}})^{\frac{4r}{3r-2}}
\|\Gamma\|_{L^{2}}^{\frac{2(r-2)}{3r-2}},
\end{eqnarray}
where we have used the following facts:
$L^{r}(\mathbb{R}^{2})\hookrightarrow
B_{\infty,2}^{\frac{-r-2}{2r}}(\mathbb{R}^{2})$ (see Lemma \ref{LIM03}) with $2<r<\infty$ and
$\|f\|_{\dot{H}^{s}}\leq C \|f\|_{H^{s}}$ for any $s>0$.
\\
Finally, by Young inequality and Gagliardo-Nirenberg inequality, one can easily check that
\begin{eqnarray}\label{ye1}
\mathcal{N}_{3} &\leq& \|\mathcal
{R}(\tau\Omega-\Omega\tau)\|_{L^{\frac{2}{2-\alpha}}} \|\Gamma
\|_{L^{\frac{2}{\alpha}}}\nonumber\\
&\lesssim&
\|\tau\Omega-\Omega\tau\|_{L^{\frac{2}{2-\alpha}}} \|\Gamma
\|_{L^{\frac{2}{\alpha}}}\nonumber\\
&\lesssim& \|\Omega\|_{L^{2}}\|\tau\|_{L^{\frac{2}{1-\alpha}}}
\|\Gamma\|_{L^{2}}^{\alpha} \|\nabla\Gamma
\|_{L^{2}}^{1-\alpha}\qquad \big(\alpha>0\big)\nonumber\\
&\lesssim& \|\Gamma-\mathcal
{R}\tau\|_{L^{2}}\|\Lambda^{\alpha}\tau\|_{L^{2}}
\|\Gamma\|_{L^{2}}^{\alpha} \|\nabla\Gamma
\|_{L^{2}}^{1-\alpha}\nonumber\\
&\leq&
\frac{1}{4}\|\nabla\Gamma\|_{L^{2}}^{2}+C\|\tau\|_{H^{\alpha}}^{\frac{2}{1+\alpha}}\|\Gamma
\|_{L^{2}}^{2}+C(\|\tau\|_{L^{2}}\|\tau\|_{H^{\alpha}})^{\frac{2}{1+\alpha}}\|\Gamma
\|_{L^{2}}^{\frac{2\alpha}{1+\alpha}}.
\end{eqnarray}
Let us remark that this is the only place in the proof of Lemma \ref{L32} where we use the main assumption  $\alpha>0$.\\
Putting estimates (\ref{t206}), (\ref{t208}) and (\ref{ye1}) into
(\ref{t205}), absorbing the dissipative term gives us the bound by
\begin{eqnarray}
\frac{d}{dt}\|\Gamma\|_{L^{2}}^{2}+\|\nabla\Gamma\|_{L^{2}}^{2}&\leq&
C(\|\Lambda^{\alpha}\tau\|_{L^{2}}^{2} +\|\nabla
u\|_{L^{2}}^{2})+C(1+\|\tau\|_{H^{\alpha}}^{\frac{2}{1+\alpha}})\|\Gamma
\|_{L^{2}}^{2}\nonumber\\
&& + C\|\nabla
u\|_{L^{2}}^{\frac{4r}{3r-2}}\|\tau\|_{L^{r}}^{\frac{4r}{3r-2}}
\|\Gamma\|_{L^{2}}^{\frac{2(r-2)}{3r-2}}+C(\|
u\|_{L^{2}}\|\tau\|_{L^{2}})^{\frac{4r}{3r-2}}
\|\Gamma\|_{L^{2}}^{\frac{2(r-2)}{3r-2}}\nonumber\\&&
+C(\|\tau\|_{L^{2}}\|\tau\|_{H^{\alpha}})^{\frac{2}{1+\alpha}}\|\Gamma
\|_{L^{2}}^{\frac{2\alpha}{1+\alpha}}\nonumber\\
&\leq& C(\|\Lambda^{\alpha}\tau\|_{L^{2}}^{2} +\|\nabla
u\|_{L^{2}}^{2})+C(1+\|\tau\|_{H^{\alpha}}^{2})\|\Gamma
\|_{L^{2}}^{2}\nonumber\\
&& + C\|\nabla u\|_{L^{2}}^{\frac{4r}{3r-2}}(\|\tau\|_{L^{r}}^{2}+
\|\Gamma\|_{L^{2}}^{2}),\nonumber
\end{eqnarray}
where the following facts have been applied:
$\frac{4r}{3r-2}\leq2,\,\frac{2(r-2)}{3r-2}\leq2,
\,\frac{2}{1+\alpha}\leq2$ and $\frac{2\alpha}{1+\alpha}\leq2$.\\
As a result, it follows that
\begin{eqnarray}\label{t209}
\frac{d}{dt}\|\Gamma\|_{L^{2}}^{2}+\|\nabla\Gamma\|_{L^{2}}^{2}
&\leq& C(\|\Lambda^{\alpha}\tau\|_{L^{2}}^{2} +\|\nabla
u\|_{L^{2}}^{2})+C(1+\|\tau\|_{H^{\alpha}}^{2})\|\Gamma
\|_{L^{2}}^{2}\nonumber\\
&& + C\|\nabla u\|_{L^{2}}^{\frac{4r}{3r-2}}(\|\tau\|_{L^{r}}^{2}+
\|\Gamma\|_{L^{2}}^{2}).
\end{eqnarray}
In order to close the above inequality, we need to establish the differential inequality of estimate of $\|\tau\|_{L^{r}}$.
Multiplying the stress tensor $\tau$ equation of (\ref{OLD}) by
$|\tau|^{r-2}\tau$ and integrating over $\mathbb{R}^{2}$ with
respect to variable $x$, it holds that
\begin{align}\label{t210}\frac{1}{r}\frac{d}{dt}\|\tau\|_{L^{r}}^{r}+\int_{\mathbb{R}^{2}}
(\Lambda^{2\alpha}\tau)\tau|\tau|^{r-2}\,dx=\int_{\mathbb{R}^{2}}
\mathcal {D}u \tau|\tau|^{r-2}\,dx,
\end{align}
where the divergence-free condition has been used.  Meanwhile it is
worth pointing out the following simple fact also has been used
$$ \int_{\mathbb{R}^{2}}{(\tau\Omega-\Omega\tau):|\tau|^{r-2}\tau\,dx}=0$$
for any $r\geq2 $ due to the symmetry of $\tau$.\\ As a matter of fact, we have
\begin{eqnarray} \label{tttt11}\int_{\mathbb{R}^{2}}{(\tau\Omega-\Omega\tau):|\tau|^{r-2}\tau\,dx}
&=&\int_{\mathbb{R}^{2}}{\tau_{ik}\Omega_{kj}|\tau|^{r-2}\tau_{ij}\,dx}
-\int_{\mathbb{R}^{2}}{\Omega_{kj}\tau_{ji}|\tau|^{r-2}\tau_{ki}\,dx}\nonumber\\
&=&\int_{\mathbb{R}^{2}}{\tau_{ik}\Omega_{kj}|\tau|^{r-2}\tau_{ij}\,dx}
-\int_{\mathbb{R}^{2}}{\Omega_{kj}\tau_{ij}|\tau|^{r-2}\tau_{ik}\,dx}\nonumber\\
&=&0,\end{eqnarray}
where we used $\tau_{ji}=\tau_{ij}$ and $\tau_{ki}=\tau_{ik}$ due to the symmetry of $\tau$.
 Here and in the sequel we adopt the Einstein convention about summation over
repeated indices.\\
According to the
following positive inequality (see \cite{DD} for instance)
\begin{eqnarray}\label{t211}
\int_{\mathbb{R}^{n}}{|h(x)|^{p-2}h(x)\Lambda^{\alpha}h(x)\,dx}\geq
\frac{2}{p}\int_{\mathbb{R}^{n}}{(\Lambda^{\frac{\alpha}{2}}|h(x)|
^{\frac{p}{2}})^{2}\,dx}\geq0
\end{eqnarray}
for any $0\leq \alpha\leq2, \, p\geq2$, we thus conclude that
\begin{eqnarray}\label{t212}\frac{d}{dt}\|\tau\|_{L^{r}}^{2}\leq C\|\mathcal {D}u\|_{L^{r}}
\|\tau\|_{L^{r}}.
\end{eqnarray}
The boundedness of the Riesz operator between $L^{q}$ spaces ($1<q<\infty$)
and the definition of $\Gamma$ allow us to show
\begin{eqnarray}\label{t213}\frac{d}{dt}\|\tau\|_{L^{r}}^{2}
&\leq& C\|\mathcal {D}u\|_{L^{r}} \|\tau\|_{L^{r}}\nonumber\\
&\leq& C\|\omega\|_{L^{r}} \|\tau\|_{L^{r}}\nonumber\\
&\leq& C(\|\Gamma\|_{L^{r}}+\|\tau\|_{L^{r}} )\|\tau\|_{L^{r}}\nonumber\\
&\leq& C\|\Gamma\|_{L^{2}}^{\frac{2}{r}}\|\nabla
\Gamma\|_{L^{2}}^{\frac{r-2}{r}}\|\tau\|_{L^{r}}+\|\tau\|_{L^{r}}^{2}
\nonumber\\
&\leq&
\frac{1}{4}\|\nabla\Gamma\|_{L^{2}}^{2}+C(\|\Gamma\|_{L^{2}}^{2}+\|\tau\|_{L^{r}}^{2}),
\end{eqnarray}
where we used the following Gagliardo-Nirenberg inequality
$$\|\Gamma\|_{L^{r}}\leq C\|\Gamma\|_{L^{2}}^{\frac{2}{r}}\|\nabla
\Gamma\|_{L^{2}}^{\frac{r-2}{r}},\quad 2<r<\infty.$$
Adding up the above estimates $(\ref{t211})$ and $(\ref{t213})$ altogether, we obtain
\begin{eqnarray}\label{t214}
\frac{d}{dt}(\|\Gamma\|_{L^{2}}^{2}+\|\tau\|_{L^{r}}^{2})+\|\nabla\Gamma\|_{L^{2}}^{2}
&\leq& C(\|\Lambda^{\alpha}\tau\|_{L^{2}}^{2} +\|\nabla
u\|_{L^{2}}^{2})+C(1+\|\tau\|_{H^{\alpha}}^{2})\|\Gamma
\|_{L^{2}}^{2}\nonumber\\
&& + C(1+\|\nabla u\|_{L^{2}}^{\frac{4r}{3r-2}})(
\|\Gamma\|_{L^{2}}^{2}+\|\tau\|_{L^{r}}^{2}).
\end{eqnarray}
Apply the above Lemma \ref{L31} and the differential form
of Gronwall inequality to (\ref{t214}) to get
$$\|\Gamma(t)\|_{L^{2}}^{2}+\|\tau\|_{L^{r}}^{2}+\int_{0}^{T}{
\|\nabla\Gamma(s)\|_{L^{2}}^{2}\,ds}\leq
C<\infty.$$ The above estimate together with the quantity $\Gamma=\omega-\mathcal {R}\tau$ and
$\|\Gamma\|_{L^{r}}\leq C\|\Gamma\|_{L^{2}}^{\frac{2}{r}}\|\nabla
\Gamma\|_{L^{2}}^{\frac{r-2}{r}}$, we find that the following is immediate
$$\int_{0}^{T}{\|\omega(t)\|_{L^{r}}^{\frac{2r}{r-2}}\,dt}\leq
C<\infty.$$ Therefore, this concludes the proof of Lemma \ref{L32}.
\end{proof}
\begin{rem}\rm
Here we want to point out that Lemma \ref{L32} still holds true for the case $\alpha=0$ when the corotational term $\tau\Omega-\Omega\tau$ is absent from the system (\ref{OLD}).
\end{rem}
With the global bound (\ref{t202}) at our disposal, we are ready to
prove the following key lemma which is valid for any $\alpha>0$.
\begin{lemma}\label{L33} Suppose that $\alpha>0$ and
for any corresponding solution $(u, \tau)$ of (\ref{OLD}), there
exist some constants $C$ such that for any $T>0$
\begin{eqnarray}\label{t215}
\|\nabla\tau(t)\|_{L^{2}}^{2}+\int_{0}^{T}{(\|\nabla
\omega\|_{L^{2}}^{2}+\|\Lambda^{\alpha}\nabla\tau\|_{L^{2}}^{2})(s)\,ds}\leq
C<\infty
\end{eqnarray}
for any $t\in[0, T].$
\end{lemma}
\begin{proof}[\textbf{Proof of  Lemma \ref{L33}}]
Applying $\nabla$ to the Oldroyd-B equations $(\ref{OLD})_{2}$ and testing the resulting equation by $\nabla\tau$, we get
\begin{eqnarray}\label{t216}
\frac{1}{2}\frac{d}{dt}\|\nabla
\tau\|_{L^{2}}^{2}+\|\Lambda^{\alpha}\nabla\tau\|_{L^{2}}^{2}
&=&\int_{\mathbb{R}^{2}}{\nabla\mathcal {D}u\cdot \nabla
\tau\,dx}-\int_{\mathbb{R}^{2}}{\nabla(u\cdot\nabla \tau): \nabla
\tau\,dx}\nonumber\\&&+\int_{\mathbb{R}^{2}}{\nabla(\tau\Omega-\Omega\tau)\cdot\nabla\tau\,dx}.
\end{eqnarray}
Multiplying the vorticity equation $(\ref{t203})$
by $ w$, noting the incompressibility condition, it leads to
\begin{eqnarray}\label{t217}
\frac{1}{2}\frac{d}{dt}\|\omega\|_{L^{2}}^{2}+\|\nabla \omega\|_{L^{2}}^{2}
&=&\int_{\mathbb{R}^{2}}{({\rm curl \,div}
(\tau))\omega\,dx}.
\end{eqnarray}
Combining (\ref{t216}) and (\ref{t217}), we arrive at
\begin{eqnarray}\label{t218}
&&\frac{1}{2}\frac{d}{dt}(\|\omega\|_{L^{2}}^{2}+\|\nabla
\tau\|_{L^{2}}^{2})+\|\nabla \omega\|_{L^{2}}^{2}+\|\Lambda^{\alpha}\nabla\tau\|_{L^{2}}^{2}\nonumber\\
&=&\int_{\mathbb{R}^{2}}{({\rm curl \,div}
(\tau))\omega\,dx}+\int_{\mathbb{R}^{2}}{\nabla\mathcal {D}u\cdot \nabla
\tau\,dx}-\int_{\mathbb{R}^{2}}{\nabla(u\cdot\nabla \tau): \nabla
\tau\,dx}\nonumber\\&&+
\int_{\mathbb{R}^{2}}{\nabla(\tau\Omega-\Omega\tau)\cdot\nabla\tau\,dx}\nonumber\\&:=&
\mathcal{K}_{1}+\mathcal{K}_{2}+\mathcal{K}_{3}+\mathcal{K}_{4}.
\end{eqnarray}
In what follows, we will deal with each term on the right-hand side
of (\ref{t218}) separately. First of all, the first two terms
$\mathcal{K}_{1}$ and  $\mathcal{K}_{2}$ are easy to handle. Indeed, integrating by
parts and taking advantage of Young inequality lead to
\begin{eqnarray}\label{t219}
\mathcal{K}_{1} &\leq& \|\nabla\tau\|_{L^{2}} \|\nabla w \|_{L^{2}}\nonumber\\
&\leq& \frac{1}{4}\|\nabla \omega\|_{L^{2}}^{2}+C\|\nabla \tau
\|_{L^{2}}^{2}.
\end{eqnarray}
The Young inequality directly gives
\begin{eqnarray}\label{t220}
\mathcal{K}_{2} &\leq& \|\nabla\tau\|_{L^{2}} \|\nabla \mathcal {D}u \|_{L^{2}}\nonumber\\
&\lesssim&\|\nabla\tau\|_{L^{2}} \|\nabla \omega \|_{L^{2}}
\nonumber\\&\leq& \frac{1}{4}\|\nabla \omega\|_{L^{2}}^{2}+C\|\nabla \tau
\|_{L^{2}}^{2}.
\end{eqnarray}
Recalling the incompressibility condition, the term $\mathcal{K}_{3}$ can be rewritten as
\begin{eqnarray}
\mathcal{K}_{3} &=& -\int_{\mathbb{R}^{2}}{\partial_{l}( u_{i}\partial_{i} \tau_{kj})\partial_{l}
\tau_{kj}\,dx}\nonumber\\&=&-\int_{\mathbb{R}^{2}}{\partial_{l} u_{i}\partial_{i} \tau_{kj}\partial_{l}
\tau_{kj}\,dx}.\nonumber
\end{eqnarray}
By H$\rm\ddot{o}$lder inequality and Gagliardo-Nirenberg inequality, we can obtain
\begin{eqnarray}\label{t221}
\mathcal{K}_{3} &\leq& C\|\nabla u\|_{L^{r}}\|\nabla \tau\|_{L^{\frac{2r}{r-1}}}^{2}
\nonumber\\&\leq&
C\|\omega\|_{L^{r}}\|\nabla \tau\|_{L^{2}}^{\frac{2(\alpha r-1)}{\alpha r}}\|\Lambda^{\alpha}\nabla\tau\|_{L^{2}}^{\frac{2}{\alpha r}}
\nonumber\\&\leq&
\frac{1}{2}\|\Lambda^{\alpha}\nabla\tau\|_{L^{2}}^{2}+C\|\omega\|_{L^{r}}^{\frac{\alpha r}{\alpha r-1}}\|\nabla \tau\|_{L^{2}}^{2},
\end{eqnarray}
where we have applied the following Gagliardo-Nirenberg inequality
$$\|\nabla \tau\|_{L^{\frac{2r}{r-1}}}\leq C \|\nabla \tau\|_{L^{2}}^{\frac{\alpha r-1}{\alpha r}}\|\Lambda^{\alpha}\nabla\tau\|_{L^{2}}^{\frac{1}{\alpha r}},\quad r>\frac{1}{\alpha}.$$
Similarly, the last term $\mathcal{K}_{4}$ can be bounded by
\begin{eqnarray}\label{t222}
\mathcal{K}_{4} &\leq& C \int_{\mathbb{R}^{2}}{|\nabla\tau|\,|\Omega|\,\nabla\tau|\,dx}+C \int_{\mathbb{R}^{2}}{|\tau|\,|\nabla\Omega|\,\nabla\tau|\,dx}
\nonumber\\&\leq&
C\|\nabla u\|_{L^{r}}\|\nabla \tau\|_{L^{\frac{2r}{r-1}}}^{2}+C\|\tau\|_{L^{\frac{2}{\alpha(1-\varepsilon)}}}
\|\nabla\Omega\|_{L^{2}}
\|\nabla\tau\|_{L^{\frac{2}{1-\alpha+\alpha\varepsilon}}}
\nonumber\\&\leq&
C\|\omega\|_{L^{r}}\|\nabla \tau\|_{L^{2}}^{\frac{2(\alpha r-1)}{\alpha r}}\|\Lambda^{\alpha}\nabla\tau\|_{L^{2}}^{\frac{2}{\alpha r}}+C\|\tau\|_{L^{\frac{2}{\alpha(1-\varepsilon)}}}
\|\nabla \omega\|_{L^{2}}
\|\nabla\tau\|_{L^{2}}^{\varepsilon}\|\Lambda^{\alpha}\nabla\tau\|_{L^{2}}^{1-\varepsilon}
\nonumber\\&\leq&
\frac{1}{4}\|\Lambda^{\alpha}\nabla\tau\|_{L^{2}}^{2}+ \frac{1}{4}\|\nabla \omega\|_{L^{2}}^{2}+C\|\omega\|_{L^{r}}^{\frac{\alpha r}{\alpha r-1}}\|\nabla \tau\|_{L^{2}}^{2}+C\|\tau\|_{L^{\frac{2}{\alpha(1-\varepsilon)}}}
^{\frac{2}{\varepsilon}}\|\nabla
\tau\|_{L^{2}}^{2},
\end{eqnarray}
where we can select $\varepsilon\in (0,\,1)$ suitably small to satisfy the following Gagliardo-Nirenberg inequality
$$\|\nabla\tau\|_{L^{\frac{2}{1-\alpha+\alpha\varepsilon}}}\leq C \|\nabla\tau\|_{L^{2}}^{\varepsilon}\|\Lambda^{\alpha}
\nabla\tau\|_{L^{2}}^{1-\varepsilon}.$$
Inserting the four previous estimates (\ref{t219}), (\ref{t220}), (\ref{t221}) and (\ref{t222}) into (\ref{t218}), ignoring the dissipative terms,
one obtains
\begin{eqnarray}\label{t223}
&&\frac{d}{dt}(\|\omega\|_{L^{2}}^{2}+\|\nabla
\tau\|_{L^{2}}^{2})+\|\nabla \omega\|_{L^{2}}^{2}+\|\Lambda^{\alpha}\nabla\tau\|_{L^{2}}^{2}\nonumber\\
&\leq&C\|\nabla \tau
\|_{L^{2}}^{2}+C\|\omega\|_{L^{r}}^{\frac{\alpha r}{\alpha r-1}}\|\nabla \tau\|_{L^{2}}^{2}+C\|\tau\|_{L^{\frac{2}{\alpha(1-\varepsilon)}}}
^{\frac{2}{\varepsilon}}\|\nabla
\tau\|_{L^{2}}^{2}.
\end{eqnarray}
Recalling the following bound obtained in Lemma \ref{L32}
$$\int_{0}^{T}{\|\omega(t)\|_{L^{r}}^{\frac{2r}{r-2}}\,dt}\leq
C<\infty,$$
it is easy to check that
$$\int_{0}^{T}{\|\omega(t)\|_{L^{r}}^{\frac{\alpha r}{\alpha r-1}}\,dt}\leq
C<\infty,$$
provided that $$r>\frac{2(1-\alpha)}{\alpha}.$$
Therefore, we choose the following $r$ in proving this Lemma \ref{L33},
$$r>\max\Big\{\frac{2(1-\alpha)}{\alpha},\,\,\,\frac{1}{\alpha}\Big\}.$$
With the aid of above observation, we can conclude the desired result by applying Gronwall inequality to (\ref{t223}).
Consequently, this concludes the proof of Lemma \ref{L33}.
\end{proof}

With Lemmas \ref{L32} and \ref{L33} at our disposal, we are now turning to give the proof of the main result as follows.
\begin{proof}[\textbf{Proof of Theorem \ref{Th1}}]
Applying operation $\Lambda^{s}$ with $s>2$ to system(\ref{OLD}) and
taking the $L^{2}$ inner product with $\Lambda^{s}u$ and
$\Lambda^{s}\tau$ respectively, adding them up, we can get
\begin{eqnarray}\label{t224}
&&\frac{1}{2}\frac{d}{dt}(\|\Lambda^{s}u(t)\|_{L^{2}}^{2}+
\|\Lambda^{s}\tau(t)\|_{L^{2}}^{2})+\|\Lambda^{s+1}
u\|_{L^{2}}^{2}+\|\Lambda^{s+\alpha}
\tau\|_{L^{2}}^{2}\nonumber\\
&\lesssim&\|\Lambda^{s}\tau(t)\|_{L^{2}} \|\Lambda^{s+1} u\|_{L^{2}}
+\int_{\mathbb{R}^{2}}{|[\Lambda^{s}, u\cdot\nabla]u\cdot
\Lambda^{s}u|\,dx}+\int_{\mathbb{R}^{2}}{|[\Lambda^{s},
u\cdot\nabla]\tau:
\Lambda^{s}\tau|\,dx}
\nonumber\\&&+\int_{\mathbb{R}^{2}}{\Lambda^{s}(\tau\Omega-\Omega\tau):\Lambda^{s}\tau\,dx}
\nonumber\\
&:=&\mathcal{ L}_{1}+\mathcal{ L}_{2}+\mathcal{ L}_{3}+\mathcal{ L}_{4}.
\end{eqnarray}
The first term can be bounded by
\begin{eqnarray}\label{t225}
\mathcal{ L}_{1}\leq \frac{1}{4}\|\Lambda^{s+1}
u\|_{L^{2}}^{2}+C\|\Lambda^{s}\tau\|_{L^{2}}^{2}.
\end{eqnarray}
To estimate $\mathcal{ L}_{2}$, $\mathcal{ L}_{3}$ and $\mathcal{ L}_{4}$, the following commutator
estimates and bilinear estimates will be used (see Kato-Ponce \cite{KP} and
Kenig-Ponce-Vega \cite{KPV})
\begin{eqnarray}\label{t226}\|[\Lambda^{s},
f]g\|_{L^{p}}\leq C(\|\nabla
f\|_{L^{p_{1}}}\|\Lambda^{s-1}g\|_{L^{p_{2}}}+\|\Lambda^{s}f\|_{L^{p_{3}}}\|g\|_{L^{p_{4}}}),
\end{eqnarray}
and
\begin{eqnarray}\label{t227}\|\Lambda^{s}(f g)\|_{L^{p}}\leq C(\| f\|_{L^{p_{1}}}
\|\Lambda^{s}g\|_{L^{p_{2}}}+\|\Lambda^{s}f\|_{L^{p_{3}}}\|g\|_{L^{p_{4}}})\end{eqnarray}
with $s>0,\,p_{2}, p_{3}\in(1, \infty)$ such that
$\frac{1}{p}=\frac{1}{p_{1}}+\frac{1}{p_{2}}=\frac{1}{p_{3}}+\frac{1}{p_{4}}.$\\
According to above commutator estimate (\ref{t226}), it follows that
\begin{eqnarray}\label{t228}
\mathcal{ L}_{2}&\leq& C \|[\Lambda^{s},
u\cdot\nabla]u\|_{L^{2}}\|\Lambda^{s}u\|_{L^{2}}\nonumber\\ &\leq& C
\|\nabla u\|_{L^{\infty}}\|\Lambda^{s}u\|_{L^{2}}^{2},
\end{eqnarray}
\begin{eqnarray}\label{t229}
\mathcal{ L}_{3}&\leq& C \|[\Lambda^{s},
u\cdot\nabla]\tau\|_{L^{2}}\|\Lambda^{s}\tau\|_{L^{2}}\nonumber\\
&\leq& C (\|\nabla
u\|_{L^{\infty}}\|\Lambda^{s}\tau\|_{L^{2}}+\|\nabla
\tau\|_{L^{\frac{2}{1-\alpha}}}\|\Lambda^{s}u\|_{L^{\frac{2}{\alpha}}})
\|\Lambda^{s}\tau\|_{L^{2}}\nonumber\\
&\leq& C (\|\nabla
u\|_{L^{\infty}}\|\Lambda^{s}\tau\|_{L^{2}}+\|\Lambda^{\alpha}\nabla
\tau\|_{L^{2}}\|\Lambda^{s}u\|_{L^{2}}^{\alpha}\|\Lambda^{s+1}
u\|_{L^{2}}^{1-\alpha}) \|\Lambda^{s}\tau\|_{L^{2}}\nonumber\\
&\leq& \frac{1}{4}\|\Lambda^{s+1} u\|_{L^{2}}^{2}+C \|\nabla
u\|_{L^{\infty}}\|\Lambda^{s}\tau\|_{L^{2}}^{2}+C\|\Lambda^{s}u\|_{L^{2}}^{2}
+\|\Lambda^{\alpha}\nabla \tau\|_{L^{2}}^{2}
\|\Lambda^{s}\tau\|_{L^{2}}^{2}.
\end{eqnarray}
Thanks to above bilinear estimate (\ref{t227}), it is easy to show
\begin{eqnarray}\label{t230}
\mathcal{L}_{4} &\leq& C \int_{\mathbb{R}^{2}}{\Lambda^{s}(\tau\Omega)\Lambda^{s}\tau\,dx}
\nonumber\\&\leq&
C\|\nabla u\|_{L^{r}}\|\Lambda^{s} \tau\|_{L^{\frac{2r}{r-1}}}^{2}+C\|\tau\|_{L^{2r}}
\|\Lambda^{s}\Omega\|_{L^{2}}
\|\Lambda^{s}\tau\|_{L^{\frac{2r}{r-1}}}
\nonumber\\&\leq&
C\|w\|_{L^{r}}\|\Lambda^{s}\tau\|_{L^{2}}^{\frac{2(\alpha r-1)}{\alpha r}}\|\Lambda^{s+\alpha}\tau\|_{L^{2}}^{\frac{2}{\alpha r}}+C\|\tau\|_{L^{2r}}
\|\Lambda^{s+1}u\|_{L^{2}}
\|\Lambda^{s}\tau\|_{L^{2}}^{\frac{\alpha r-1}{\alpha r}}\|\Lambda^{s+\alpha}\tau\|_{L^{2}}^{\frac{1}{\alpha r}}
\nonumber\\&\leq&
\frac{1}{4}\|\Lambda^{s+\alpha}\tau\|_{L^{2}}^{2}+ \frac{1}{4}\|\Lambda^{s+1}u\|_{L^{2}}^{2}+C\|w\|_{L^{r}}^{\frac{\alpha r}{\alpha r-1}}\|\Lambda^{s}\tau\|_{L^{2}}^{2}+C\|\tau\|_{L^{2r}}
^{\frac{2\alpha r}{\alpha r-1}}\|\Lambda^{s}
\tau\|_{L^{2}}^{2}.
\end{eqnarray}
Here we want to state that $r>\frac{1}{\alpha}$ as in proving $\mathcal{K}_{4}$.\\
Substituting all the preceding estimates into (\ref{t224}), one reaches
\begin{eqnarray}\label{t231}
&& \frac{d}{dt}(\|\Lambda^{s}u(t)\|_{L^{2}}^{2}+
\|\Lambda^{s}\tau(t)\|_{L^{2}}^{2})+\|\Lambda^{s+1}
u\|_{L^{2}}^{2}+\|\Lambda^{s+\alpha}
\tau\|_{L^{2}}^{2}\nonumber\\
&\lesssim&(1+\|\Lambda^{\alpha}\nabla \tau\|_{L^{2}}^{2}
+\|w\|_{L^{r}}^{\frac{\alpha r}{\alpha r-1}}+\|\tau\|_{L^{2r}}
^{\frac{2\alpha r}{\alpha r-1}})(\|\Lambda^{s}u(t)\|_{L^{2}}^{2}+
\|\Lambda^{s}\tau(t)\|_{L^{2}}^{2})\nonumber\\&&+ \|\nabla
u\|_{L^{\infty}}(\|\Lambda^{s}u(t)\|_{L^{2}}^{2}+
\|\Lambda^{s}\tau(t)\|_{L^{2}}^{2}).
\end{eqnarray}
Thus in order to deal with $\|\nabla u\|_{L^{\infty}}$, we turn our attention to the following standard logarithmic Sobolev inequality (see, e.g.,
\cite{kOT})
\begin{eqnarray}\|\nabla f\|_{L^{\infty}(\mathbb{R}^{2})}
\leq C\Big(1+\|f\|_{L^{2}(\mathbb{R}^{2})}+ \|\nabla (\nabla\times
f)\|_{L^{2}(\mathbb{R}^{2})} \log\big(e+\|\Lambda^{s}
f\|_{{L}^{2}(\mathbb{R}^{2})}\big)\Big),\nonumber
\end{eqnarray}
for all divergence-free functions $f$ with $f\in{H}^{s}(\mathbb{R}^{2})$ for any $s>2$.
It should be noted that the proof of this inequality is not particularly difficult via the Besov techniques but will not be reproduced here.\\
Applying above logarithmic Sobolev inequality to (\ref{t231}), it implies that
\begin{eqnarray}\label{t232}
&& \frac{d}{dt}(\|\Lambda^{s}u(t)\|_{L^{2}}^{2}+
\|\Lambda^{s}\tau(t)\|_{L^{2}}^{2})+\|\Lambda^{s+1}
u\|_{L^{2}}^{2}+\|\Lambda^{s+\alpha}
\tau\|_{L^{2}}^{2}\nonumber\\
&\lesssim&\Big(1+\|u\|_{L^{2} }+ \|\nabla \omega\|_{L^{2} }
\log\big(e+\|\Lambda^{s}u(t)\|_{L^{2}}^{2}+\|\Lambda^{s}\tau(t)\|_{L^{2}}^{2}\big)\Big)(\|\Lambda^{s}u(t)\|_{L^{2}}^{2}+
\|\Lambda^{s}\tau(t)\|_{L^{2}}^{2})\nonumber\\&&+(1+\|\Lambda^{\alpha}\nabla
\tau\|_{L^{2}}^{2}+\|\omega\|_{L^{r}}^{\frac{\alpha r}{\alpha r-1}}+\|\tau\|_{L^{2r}}
^{\frac{2\alpha r}{\alpha r-1}})(\|\Lambda^{s}u(t)\|_{L^{2}}^{2}+
\|\Lambda^{s}\tau(t)\|_{L^{2}}^{2})
.\nonumber
\end{eqnarray}
By the estimates obtained in Lemmas \ref{L32} and \ref{L33} and the standard Log-Gronwall argument, it follows from this inequality that
$$\|\Lambda^{s}u(t)\|_{L^{2}}^{2}+
\|\Lambda^{s}\tau(t)\|_{L^{2}}^{2}+\int_{0}^{T}{\big(\|\Lambda^{s+1}
u\|_{L^{2}}^{2}+\|\Lambda^{s+\alpha}
\tau\|_{L^{2}}^{2}\big)(s)\,ds}\leq C<\infty,$$
which is the desired global bounds in Theorem \ref{Th1}.
Moreover, the uniqueness is easy since the velocity and the stress
tensor are both Lipschitz.
Indeed, let $(u,\,\tau,\,\pi)$ and
$(\widetilde{u},\,\widetilde{\tau},\,\widetilde{\pi})$ be two
solutions of $(\ref{OLD})$ with the same initial data. We denote
$V=u-\widetilde{u}$, $W=\tau-\widetilde{\tau}$ and
$P=\pi-\widetilde{\pi}$. Then  the difference pair $(V,\,W,\,P)$
satisfies
\begin{equation}\label{DOLD}
\left\{\aligned
&\partial_{t}V+(u\cdot\nabla)V-\nu\Delta V+\nabla P= \nabla\cdot W-(V\cdot\nabla )\widetilde{u},\\
&\partial_{t}W+(u\cdot\nabla)W+(-\Delta)^{\alpha} W+F(V,W)= \mathcal {D}V-(V\cdot\nabla )\widetilde{\tau},\\
&\nabla\cdot u=0,\\
& W(0)=V(0)=0,
\endaligned \right.
\end{equation}
where $F(V,W)=W\Omega-\widetilde{\Omega}W+\frac{1}{2}\big(\widetilde{\tau} \nabla V -\widetilde{\tau} (\nabla V)^{T}+\nabla V \tau -(\nabla V)^{T} \tau \big)$.\\
Taking $L^{2}$-product of the above equations $(\ref{DOLD})_{1}$,
$(\ref{DOLD})_{2}$ with $V$  and $W$, respectively, it gives rise to
\begin{eqnarray}\nonumber
&&\frac{1}{2}\frac{d}{dt}(\|V(t)\|_{L^{2}}^{2}+\|W(t)\|_{L^{2}}^{2})+\|\nabla
 V\|_{L^{2}}^{2}+\|\Lambda^{\alpha} W \|_{L^{2}}^{2}\nonumber\\
&\lesssim& \|\nabla
\widetilde{u}\|_{L^{\infty}}\|V\|_{L^{2}}^{2}+\|
\widetilde{\tau}\|_{L^{\infty}}^{2}\|W\|_{L^{2}}^{2}+\|
{\tau}\|_{L^{\infty}}^{2}\|W\|_{L^{2}}^{2}+\|\nabla
\widetilde{\tau}\|_{L^{\infty}}\|V\|_{L^{2}}\|W\|_{L^{2}}\nonumber\\
&\lesssim&(\|\nabla \widetilde{u}\|_{L^{\infty}}+\|\nabla
\widetilde{\tau}\|_{L^{\infty}}+\|
{\tau}\|_{L^{\infty}}^{2}+\|
\widetilde{\tau}\|_{L^{\infty}}^{2})(\|V\|_{L^{2}}^{2}+\|W\|_{L^{2}}^{2}),\nonumber
\end{eqnarray}
where the following identity was used
$$\int_{\mathbb{R}^{2}}{(\nabla\cdot W)\cdot V
\,dx}+\int_{\mathbb{R}^{2}}{\mathcal {D}V:W\,dx}=0. $$ Thanks to the global $H^{s}$ bound with $s>2$, we know
$$\int_{0}^{T}{(\|\nabla \widetilde{u}\|_{L^{\infty}}+\|\nabla
\widetilde{\tau}\|_{L^{\infty}}+\|
{\tau}\|_{L^{\infty}}^{2}+\|
\widetilde{\tau}\|_{L^{\infty}}^{2})(s)\,ds}\leq C<\infty,\quad for\,\,
any\,\, 0\leq T<\infty.$$ Gronwall inequality implies the following estimate
\begin{eqnarray}\nonumber
&&\|V(t)\|_{L^{2}}^{2}+\|W(t)\|_{L^{2}}^{2}+\int_{0}^{T}{(\|\nabla
 V(s)\|_{L^{2}}^{2}+\|\Lambda^{\alpha} W (s)\|_{L^{2}}^{2})\,ds}\nonumber\\
&\leq& ( \|V(0)\|_{L^{2}}^{2}+\|W(0)\|_{L^{2}}^{2}){\rm
exp}\Big[\int_{0}^{T}{(\|\nabla \widetilde{u}\|_{L^{\infty}}+\|\nabla
\widetilde{\tau}\|_{L^{\infty}}+\|
{\tau}\|_{L^{\infty}}^{2}+\|
\widetilde{\tau}\|_{L^{\infty}}^{2})(s)\,ds}\Big].\nonumber
\end{eqnarray}
Therefore, the
desired uniqueness is an easy consequence of
$\|V(0)\|_{L^{2}}^{2}=\|W(0)\|_{L^{2}}^{2}=0$. Consequenltly, the statements in Theorem
\ref{Th1} are proved.
Thus, this completes the proof of Theorem \ref{Th1}.
\end{proof}

\vskip .3in
\appendix
\section{some a priori estimates and the proof of Theorem \ref{Th2}}
\subsection{some a priori estimates}
 As discussed in the introduction, the global regularity to the
  system (\ref{OLD}) with $\alpha=\eta=0$ is open and quite
  challenging. Thus, in this Appendix we provide some a priori
   estimates to the following system which may be useful and
   of interest for future improvement,
 \begin{equation}\label{Open}
\left\{\aligned
&\partial_{t}u+(u\cdot\nabla)u-\Delta u+\nabla \pi=\nabla\cdot \tau,\\
&\partial_{t}\tau+(u\cdot\nabla)\tau=\mathcal {D}u,\\
&\nabla\cdot u=0,\\
&u(x,0)=u_{0}(x),\,\tau(x,0)=\tau_{0}(x).
\endaligned \right.
\end{equation}
 Here, it is worthy to remark that the global bound (\ref{t202}) is enough for us to prove Theorem \ref{Th1}. In other words, the global a priori estimates stated in this Appendix will not be used in proving Theorem \ref{Th1}.

The global bound (\ref{t202}) allows us to improve the integrability
of the vorticity $w$ as follows. Now we would like
to state and prove the first lemma.
\begin{lemma}\label{LIM01}
 Let $(u, \tau)$ be any corresponding
solution of (\ref{Open}). For any $2<p,\,\,q<\infty$, there exist some constants $C$
such that for any $T>0$
\begin{eqnarray}\label{IMt01}
\int_{0}^{T}{\|\nabla u(t)\|_{L^{p}}^{q}\,dt}\leq C<\infty.
\end{eqnarray}
\end{lemma}
To prove this lemma, we recall some properties involving the heat
operator. Now we set ($n$ is the space dimension)
$$e^{t\Delta}f=H(t,x)\ast f,\quad H(t,x)=(4\pi
t)^{-\frac{n}{2}}e^{-\frac{|x|^{2}}{4t}}.$$ The following estimate
is a direct consequence of the Young inequality.
\begin{eqnarray}\label{IMt02}\|\nabla^{k}e^{t\Delta}f\|_{L^{q}}\leq C
t^{-\frac{k}{2}-\frac{n}{2}(\frac{1}{p}-\frac{1}{q})}
\|f\|_{L^{p}},\end{eqnarray}
 for every
$1\leq p\leq q\leq\infty$ and integer $k\geq0$.\\
To prove the lemma, we also need the following Maximal $L_t^{q}L_x^{p}$
regularity for the heat kernel. The proof is omitted here and we can
find the detailed proof in \cite{LPG} for example.
\begin{Pros}\label{Pr1}
The operator $A$ defined by
$$Af(x,t)\triangleq \int_{0}^{t}{e^{(t-s)\Delta}\Delta f(s,x)\,ds}$$
is bounded from $L^{p}(0,T; L^{q}(\mathbb{R}^{n}))$ to $L^{p}(0,T;
L^{q}(\mathbb{R}^{n}))$ for very $(p,q)\in (1,\infty)\times
(1,\infty)$ and $T\in(0,\infty]$.
\end{Pros}
\begin{proof}[\textbf{Proof of  Lemma \ref{LIM01}}]
Now let us start to prove Lemma \ref{LIM01} with the help of Proposition \ref{Pr1}.
Thanks to the divergence-free condition, we can deduce from the first equation of (\ref{OLD}) that
$$-\pi=\frac{{\rm div}{\rm div}}{-\Delta}\big(\tau-u\otimes u\big)\triangleq
\mathcal {\widetilde{R}}\big(\tau-u\otimes u\big).$$ Hence, the
first equation of (\ref{OLD}) can be rewritten as
\begin{eqnarray}
\partial_{t}u-\Delta u=\big({\rm div}+\nabla \mathcal {\widetilde{R}}\big)
\big(\tau-u\otimes u\big).\nonumber\end{eqnarray}
Applying operator $\nabla$ to above equality, one arrives at
\begin{eqnarray}\label{IMt03}
\partial_{t}\nabla u-\Delta \nabla u=\nabla\big({\rm div}+\nabla
\mathcal {\widetilde{R}}\big)
\big(\tau-u\otimes u\big).\end{eqnarray}
By Duhamel's Principle, the
velocity $\nabla u$ can be solved by
\begin{eqnarray}\label{IMt04}
\nabla u(x,t)=e^{t\Delta}\nabla
u_{0}(x)-\int_{0}^{t}{e^{(t-s)\Delta}
\Delta\underbrace{(-\Delta)^{-1}\nabla\big({\rm div}+\nabla \mathcal
{\widetilde{R}}\big)\big(\tau-u\otimes u\big)}(x,s)\,ds}.
\end{eqnarray}
Note that $\|\Gamma\|_{L^{2}}\leq C$ and $\Gamma=\omega-\mathcal {R}\tau$, we can verify that
$$\|\omega\|_{L^{2}}\leq C,$$
which together with the basic energy estimate (\ref{t201}) implies
$$\|u\|_{L^{p}}\leq C,\quad \mbox{for any }\,\,2<p<\infty.$$
By (\ref{IMt02}) and Proposition \ref{Pr1}, one can conclude from (\ref{IMt04}) that
\begin{eqnarray}&&
\|\nabla u\|_{L_{T}^{q}L_{x}^{p}}\nonumber\\&\leq&
\|e^{t\Delta}\nabla u_{0}\|_{L_{T}^{q}L_{x}^{p}}+
\big\|\int_{0}^{t}{e^{(t-s)\Delta}\Delta\underbrace{(-\Delta)^{-1}
\nabla\big({\rm div}+\nabla \mathcal
{\widetilde{R}}\big)\big(\tau-u\otimes
u\big)}(x,s)\,ds}\big\|_{L_{T}^{q}L_{x}^{p}} \nonumber\\&\leq&
C\|H(t,x)\|_{L_{T}^{q}L_{x}^{1}}\|\nabla u_{0}\|_{L_{x}^{p}}+C
\big\|{(-\Delta)^{-1}\nabla\big({\rm div}+\nabla
\mathcal {\widetilde{R}}\big)\big(\tau-u\otimes u\big)}\big\|_{L_{T}^{q}L_{x}^{p}}\nonumber\\
&\leq& C\|\nabla u_{0}\|_{L_{x}^{p}}+C
\big\|\tau-u\otimes u\big\|_{L_{T}^{q}L_{x}^{p}}\nonumber\\
&\leq& C+CT^{\frac{1}{q}}(\|u \|_{L_{T}^{\infty}L_{x}^{2p}}^{2}+\|\tau \|_{L_{T}^{\infty}L_{x}^{p}})\nonumber\\ &\leq& C<\infty,\nonumber
\end{eqnarray}
where we have used the boundedness of the Calderon-Zygmund operator
between the $L^{p}$ ($1<p<\infty$) space in the third inequality.\\
Thus, this concludes the proof of Proposition \ref{Pr1}.
\end{proof}
The next lemma is concerned with the quantity $\Gamma$ which reads as
\begin{lemma}\label{LIM02}
For any corresponding
solution $(u, \tau)$ of (\ref{Open}) and any $2<p,\,\,q<\infty$, there exist some constants $C$
such that for any $T>0$
\begin{eqnarray}\label{IMt05}
\int_{0}^{T}{\|\nabla \Gamma(t)\|_{L^{p}}^{q}\,dt}\leq C<\infty.
\end{eqnarray}
\end{lemma}
\begin{proof}[\textbf{Proof of  Lemma \ref{LIM02}}]
By Lemma \ref{LIM01}, we conclude the following estimate
$$\|\Gamma\|_{L_{T}^{q}L_{x}^{p}}\leq C\|\omega\|_{L_{T}^{q}L_{x}^{p}}+C\|\tau\|_{L_{T}^{q}L_{x}^{p}}\leq C,\quad \mbox{for any }\,\,2<p,\,\,q<\infty.$$
In the case $\alpha=\eta=0$, the combined quantity $\Gamma$ satisfies
\begin{eqnarray}\partial_{t}\Gamma+(u\cdot\nabla)\Gamma-\Delta \Gamma
=[\mathcal {R},\,u\cdot\nabla]\tau-\mathcal {R}\mathcal
{D}u.\nonumber\end{eqnarray} Taking into account the divergence-free
condition and applying operator $\nabla$, we thus get
\begin{eqnarray}\label{IMt06}\partial_{t}\nabla\Gamma-\Delta \nabla\Gamma =
\nabla\Big(\nabla\cdot[\mathcal {R},\,u]\tau-\mathcal {R}\mathcal
{D}u-\nabla\cdot(u\Gamma)\Big).\end{eqnarray} Again using Duhamel's
Principle, we can rewrite $\nabla\Gamma$ as follows
\begin{eqnarray}
\nabla\Gamma(x,t)=e^{t\Delta}\nabla\Gamma_{0}(x)-\int_{0}^{t}{e^{(t-s)\Delta}
\Delta\underbrace{(-\Delta)^{-1}\nabla\Big(\nabla\cdot[\mathcal
{R},\,u]\tau-\mathcal {R}\mathcal
{D}u-\nabla\cdot(u\Gamma)}\Big)(x,s)\,ds}.\nonumber
\end{eqnarray}
By means of (\ref{IMt02}) as well as  Proposition \ref{Pr1} again,
it follows that
\begin{eqnarray}&&
\|\nabla \Gamma\|_{L_{T}^{q}L_{x}^{p}}\nonumber\\&\leq&
\|e^{t\Delta}\nabla \Gamma_{0}\|_{L_{T}^{q}L_{x}^{p}}+
\big\|\int_{0}^{t}{e^{(t-s)\Delta}\Delta\underbrace{(-\Delta)^{-1}\nabla\Big(\nabla\cdot[\mathcal
{R},\,u]\tau-\mathcal {R}\mathcal
{D}u-\nabla\cdot(u\Gamma)}\Big)(x,s)\,ds}\big\|_{L_{T}^{q}L_{x}^{p}}
\nonumber\\&\leq& C\|H(t,x)\|_{L_{T}^{q}L_{x}^{1}}\|\nabla
\Gamma_{0}\|_{L_{x}^{p}}+C
\big\|(-\Delta)^{-1}\nabla\Big(\nabla\cdot[\mathcal
{R},\,u]\tau-\mathcal {R}\mathcal
{D}u-\nabla\cdot(u\Gamma)\Big)\big\|_{L_{T}^{q}L_{x}^{p}}\nonumber\\
&\leq& C\|\nabla \Gamma_{0}\|_{L_{x}^{p}}+C \big\|[\mathcal
{R},\,u]\tau\big\|_{L_{T}^{q}L_{x}^{p}}+C
\big\|u\Gamma\big\|_{L_{T}^{q}L_{x}^{p}}+C
\big\|u\big\|_{L_{T}^{q}L_{x}^{p}}\nonumber\\
&\leq& C+CT^{\frac{1}{q}}\big\|[\mathcal {R},\,u]\tau\big\|_{L_{T}^{\infty}L_{x}^{p}}+C\|u \|_{L_{T}^{\infty}L_{x}^{2p}}\|\Gamma \|_{L_{T}^{q}L_{x}^{2p}}+CT^{\frac{1}{q}}\|u \|_{L_{T}^{\infty}L_{x}^{p}}\nonumber\\ &\leq& C<\infty,\nonumber
\end{eqnarray}
where we have used the following two facts:\\
(1) the boundedness of the Calderon-Zygmund operator
between the $L^{p}$ space for any $1<p<\infty$.\\
(2) the following estimate
\begin{eqnarray}
\big\|[\mathcal {R},\,u]\tau\big\|_{L^{p}}&\leq& C\big\|[\mathcal {R},\,u]\tau\big\|_{B_{p,\,2}^{0}}\quad \Big(B_{p,\,2}^{0}\hookrightarrow L^{p}\,\,\,\mbox{for}\,\,\,2\leq p<\infty \Big)
\nonumber\\ &\leq&
C\big\|[\mathcal {R},\,u]\tau\big\|_{B_{2,\,2}^{1-\frac{2}{p}}}\nonumber\\ &\approx&
\big\|[\mathcal {R},\,u]\tau\big\|_{H^{1-\frac{2}{p}}}
\nonumber\\ &\leq&
C(\|\nabla
u\|_{L^{2}}\|\tau\|_{B_{\infty,2}^{\frac{-2}{p}}}+\|
u\|_{L^{2}}\|\tau\|_{L^{2}})\nonumber\\ &\leq&
C(\|w\|_{L^{2}}\|\tau\|_{L^{p+\epsilon}}+\|
u\|_{L^{2}}\|\tau\|_{L^{2}})
\nonumber\\ &\leq&C,
\end{eqnarray}
where we used $L^{p+\epsilon}(\mathbb{R}^{2})\hookrightarrow
B_{\infty,2}^{\frac{-2}{p}}(\mathbb{R}^{2})$ (see Lemma \ref{LIM03}) for some $0<\epsilon<\infty$.\\
Therefore we conclude the proof of Lemma \ref{LIM02}.
\end{proof}
Finally, we state an embedding inequality which has been used in above proof.
\begin{lemma}\label{LIM03}
 Let $s>\frac{2}{p}$ with $1<p\leq \infty$. Then the following embedding inequality holds
\begin{eqnarray}
L^{p}(\mathbb{R}^{2})\hookrightarrow
B_{\infty,r}^{-s}(\mathbb{R}^{2}),\quad \mbox{for}\,\,\,1\leq r\leq\infty.
\end{eqnarray}
\end{lemma}
\begin{proof}[\textbf{Proof of  Lemma \ref{LIM03}}]
The classical Bernstein inequality (see for instance \cite{CM}) allows us to show that
for any $j\geq -1$
$$\|\Delta_{j}f\|_{L^{\infty}}\leq C2^{\frac{2}{p}j}\|\Delta_{j}f\|_{L^{p}}.$$
Consequently, for $1\leq r<\infty$ (indeed the case $r=\infty$ is more simpler), we have
$$\sum_{j=-1}^{\infty}2^{-jsr}\|\Delta_{j}f\|_{L^{\infty}}^{r}\leq C\sum_{j=-1}^{\infty}2^{-(s-\frac{2}{p})rj}\|\Delta_{j}f\|_{L^{p}}^{r}
\leq C\sum_{j=-1}^{\infty}2^{-(s-\frac{2}{p})rj}\|f\|_{L^{p}}^{r}.$$
According to $s>\frac{2}{p}$, one has
$$\sum_{j=-1}^{\infty}2^{-jsr}\|\Delta_{j}f\|_{L^{\infty}}^{r}\leq
C\|f\|_{L^{p}}^{r}.$$
Based on the definition of space $B_{\infty,r}^{-s}$, we conclude the proof.
\end{proof}
\vskip .2in
\subsection{The proof of Theorem \ref{Th2}}
In this subsection, we give the details of the proof of Theorem \ref{Th2}. In this subsection, we only consider the case
$1<\gamma\leq \frac{4}{3}$.
As a matter of fact, it is strongly
believed that the diffusion term is always good term and the larger the
power $\gamma$ is, the better effects it produces. Therefore, the case $\gamma>1$ and arbitrarily close to one is the main focus of this paper.
We remark that the case $\gamma>\frac{4}{3}$ is even more easier.
To begin with, we can easily derive from the system (\ref{FROLD}) that
\begin{lemma}\label{FRL31}
For any corresponding solution $(u, \tau)$ of (\ref{FROLD}), there
exist some constants $C$ such that for any $T>0$
\begin{eqnarray}\label{FRt201}
\|u(t)\|_{L^{2}}^{2}+\|\tau(t)\|_{L^{2}}^{2}+\int_{0}^{T}{\|\Lambda^{\gamma}
u(s)\|_{L^{2}}^{2}\,ds}\leq C<\infty
\end{eqnarray}
for any $t\in[0, T].$
\end{lemma}

Apply operator ${\rm curl}$
to the equation $(\ref{FROLD})_{1}$ to obtain the vorticity $\omega$
equation as follows
\begin{eqnarray}\label{FRt203}\partial_{t}\omega+(u\cdot\nabla)\omega+\Lambda^{2\gamma} \omega={\rm curl \,div } (\tau).\end{eqnarray}
Taking the operator
$\mathcal {R}_{\gamma}=\Lambda^{-2\gamma}{\rm curl \,div }$
and applying the operator $\mathcal {R}_{\gamma}$ to equation $(\ref{FROLD})_{2}$, we have
\begin{eqnarray}\label{FRZZtt}\partial_{t}\mathcal {R}_{\gamma}\tau+(u\cdot\nabla)\mathcal {R}_{\gamma}\tau=\mathcal {R}_{\gamma}\mathcal {D}u-[\mathcal {R}_{\gamma},\,u\cdot\nabla]\tau-\mathcal {R}_{\gamma}(\tau\Omega-\Omega\tau).\end{eqnarray}
Denoting the combined quantity $G=\omega-\mathcal {R}_{\gamma}\tau$, it is easy to verify that $G$ obeys
\begin{eqnarray}\label{FRt204}\partial_{t}G+(u\cdot\nabla)G+\Lambda^{2\gamma} G
=[\mathcal {R}_{\gamma},\,u\cdot\nabla]\tau+ \mathcal
{R}_{\gamma}(\tau\Omega-\Omega\tau)-\mathcal {R}_{\gamma}\mathcal
{D}u.\end{eqnarray}

Now we can conclude the following result.
\begin{lemma}\label{FRL32}
Assume that $1<\gamma\leq\frac{4}{3}$ and any $2<r<\infty$. For any corresponding
solution $(u, \tau)$ of (\ref{FROLD}), there exist some constants $C$
such that for any $T>0$
\begin{eqnarray}\label{FRt202}
\|G(t)\|_{L^{2}}^{2}+\|\tau(t)\|_{L^{r}}^{2}+\int_{0}^{T}{\|\Lambda^{\gamma}G(s)
\|_{L^{2}}^{2}\,ds}\leq C<\infty
\end{eqnarray}
for any $t\in[0, T].$
\end{lemma}
\begin{proof}[\textbf{Proof of  Lemma \ref{FRL32}}]
Taking the inner product of (\ref{FRt204}) with $G$, we obtain, after integration by parts
\begin{eqnarray}\label{FRt205}
\frac{1}{2}\frac{d}{dt}\|G\|_{L^{2}}^{2}+\|\Lambda^{\gamma}G\|_{L^{2}}^{2}=\mathcal{J}_{1}
+\mathcal{J}_{2}+\mathcal{J}_{3},
\end{eqnarray}
where
$$\mathcal{J}_{1}=-\int_{\mathbb{R}^{2}}{\mathcal {R}_{\gamma}\mathcal
{D}u G\,dx},\quad \mathcal{J}_{2}=\int_{\mathbb{R}^{2}}{[\mathcal
{R}_{\gamma},\,u\cdot\nabla]\tau G\,dx},$$
$$\mathcal{J}_{3}=\int_{\mathbb{R}^{2}}{\mathcal
{R}_{\gamma}(\tau\Omega-\Omega\tau)G\,dx}.$$
The first term admits the following estimate for any $1<\gamma\leq\frac{3}{2}$
\begin{eqnarray}\label{FRt206}
\mathcal{J}_{1} &\leq& \|\Lambda^{3-2\gamma}u\|_{L^{2}}\|G
\|_{L^{2}}\nonumber\\
&\lesssim& (\|u\|_{L^{2}}+\|\Lambda^{\gamma}
u\|_{L^{2}})\|G
\|_{L^{2}}.
\end{eqnarray}
Thanks to the Young inequality and the Gagliardo-Nirenberg inequality, one can show that for any $1<\gamma\leq\frac{4}{3}$
\begin{eqnarray}\label{FRye1}
\mathcal{J}_{3} &\leq& \|\mathcal
{R}_{\gamma}(\tau\Omega-\Omega\tau)\|_{L^{2}} \|G
\|_{L^{2}}\nonumber\\
&\lesssim& \|\tau\Omega-\Omega\tau\|_{L^{\frac{2}{2\gamma-1}}} \|G\|_{L^{2}} \nonumber\\
&\lesssim& \|\tau\|_{L^{\frac{2}{3(\gamma-1)}}} \|\Omega\|_{L^{\frac{2}{2-\gamma}}} \|G\|_{L^{2}}
\nonumber\\
&\lesssim&  (\|\tau\|_{L^{2}}+\|\tau\|_{L^{r}}) \|\nabla u\|_{L^{\frac{2}{2-\gamma}}} \|G\|_{L^{2}}\qquad \Big(r>\frac{2}{3(\gamma-1)}\Big)
\nonumber\\
&\lesssim& (\|\tau\|_{L^{2}}+\|\tau\|_{L^{r}}) \|\Lambda^{\gamma} u\|_{L^{2}} \|G\|_{L^{2}}
\nonumber\\
&\leq&
C+C(1+ \|\Lambda^{\gamma} u\|_{L^{2}}^{2})(\|G
\|_{L^{2}}^{2}+\|\tau\|_{L^{r}}^{2}).
\end{eqnarray}
By the following commutator estimate (see Proposition 4.2 in \cite{WUXUE})
\begin{eqnarray}\label{FRt207}\|[\mathcal {R}_{\gamma},\,u\cdot\nabla]f\|_{B_{2,2}^{0}}\leq C\|\nabla u\|_{L^{2}}(\|f\|_{B_{\infty,2}^{2-2\gamma}}+\|f\|_{L^{2}}),\quad 1<\gamma<\frac{3}{2}\end{eqnarray} for any smooth
divergence-free vector field, it follows that
\begin{eqnarray}\label{FRt208}
\mathcal{J}_{2} &=& \int_{\mathbb{R}^{2}}{[\mathcal
{R}_{\gamma},\,u\cdot\nabla]\tau G\,dx}\nonumber\\
&\lesssim& \|[\mathcal
{R}_{\gamma},\,u\cdot\nabla]\tau\|_{B_{2,2}^{0}}
\|G\|_{L^{2}}\nonumber\\
&\lesssim& C\|\nabla u\|_{L^{2}}(\|\tau\|_{B_{\infty,2}^{2-2\gamma}}+\|\tau\|_{L^{2}})\|G\|_{L^{2}}\nonumber\\
&\lesssim& C\|\nabla u\|_{L^{2}}(\|\tau\|_{L^{r}}+\|\tau\|_{L^{2}})\|G\|_{L^{2}}\qquad \Big(r>\frac{1}{\gamma-1}\Big)
\nonumber\\
&\lesssim& C(\|u\|_{L^{2}}+\|\Lambda^{\gamma} u\|_{L^{2}})(\|\tau\|_{L^{r}}+\|\tau\|_{L^{2}})\|G\|_{L^{2}}\nonumber\\
&\lesssim& C(\|u\|_{L^{2}}+\|\Lambda^{\gamma} u\|_{L^{2}}+\|\tau\|_{L^{2}}^{2})(\|G\|_{L^{2}}^{2}+\|\tau\|_{L^{r}}^{2}).
\end{eqnarray}
Inserting the above estimates (\ref{FRt206}), (\ref{FRt208}) and (\ref{FRye1}) into
(\ref{FRt205}), absorbing the dissipative term, we thus deduce 
\begin{eqnarray}\label{FRt209}
\frac{d}{dt}\|G\|_{L^{2}}^{2}+\|\Lambda^{\gamma}G\|_{L^{2}}^{2}
&\leq& C(1+\|u\|_{L^{2}}+\|\Lambda^{\gamma} u\|_{L^{2}}+\|\tau\|_{L^{2}}^{2})(\|G\|_{L^{2}}^{2}+\|\tau\|_{L^{r}}^{2}).
\end{eqnarray}
In order to close the above inequality, we need to establish the differential inequality of estimate of $\|\tau\|_{L^{r}}$.
Multiplying the stress tensor $\tau$ equation of (\ref{FROLD}) by
$|\tau|^{r-2}\tau$ and integrating over $\mathbb{R}^{2}$ yield
\begin{eqnarray}\label{FRt210}\frac{1}{r}\frac{d}{dt}\|\tau\|_{L^{r}}^{r}&=&\int_{\mathbb{R}^{2}}
\mathcal {D}u \tau|\tau|^{r-2}\,dx
\nonumber\\
&\leq& C\|\mathcal {D}u\|_{L^{r}} \|\tau\|_{L^{r}}^{r-1}\nonumber\\
&\leq& C\|\omega\|_{L^{r}} \|\tau\|_{L^{r}}^{r-1}\nonumber\\
&\leq& C(\|G\|_{L^{r}}+\|\mathcal
{R}_{\gamma}\tau\|_{L^{r}} )\|\tau\|_{L^{r}}^{r-1}\nonumber\\
&\leq& C(\|G\|_{L^{2}}+\|\Lambda^{\gamma}G\|_{L^{2}}+\|\tau\|_{L^{2}}+\|\tau\|_{L^{r}})
\|\tau\|_{L^{r}}^{r-1},
\end{eqnarray}
which further implies that
\begin{eqnarray}\label{FRt213}\frac{d}{dt}\|\tau\|_{L^{r}}^{2}
&\leq& C(\|G\|_{L^{2}}+\|\Lambda^{\gamma}G\|_{L^{2}}+\|\tau\|_{L^{2}}+\|\tau\|_{L^{r}})
\|\tau\|_{L^{r}}\nonumber\\
&\leq&
\frac{1}{4}\|\Lambda^{\gamma}G
\|_{L^{2}}^{2}+C+C(\|G\|_{L^{2}}^{2}+\|\tau\|_{L^{r}}^{2}).
\end{eqnarray}
Adding up the above estimates $(\ref{FRt209})$ and $(\ref{FRt213})$ altogether, we obtain
\begin{eqnarray}
\frac{d}{dt}(\|G\|_{L^{2}}^{2}+\|\tau\|_{L^{r}}^{2})+\|\Lambda^{\gamma}G\|_{L^{2}}^{2}
&\leq& C+C(1+\|u\|_{L^{2}}+\|\Lambda^{\gamma} u\|_{L^{2}}^{2}+\|\tau\|_{L^{2}}^{2})\nonumber\\&&\times(\|G\|_{L^{2}}^{2}+\|\tau\|_{L^{r}}^{2}).\nonumber
\end{eqnarray}
The classical Gronwall inequality allows us to obtain
$$\|G(t)\|_{L^{2}}^{2}+\|\tau\|_{L^{r}}^{2}+\int_{0}^{T}{
\|\Lambda^{\gamma}G(s)\|_{L^{2}}^{2}\,ds}\leq
C<\infty.$$
 Therefore, this concludes the proof of Lemma \ref{FRL32}.
\end{proof}

With Lemma \ref{FRL32} in hand, we can show
\begin{lemma}\label{FRL33}
Under the assumptions of Lemma \ref{FRL32}, there exist some constants $C$
such that for any $T>0$
\begin{eqnarray}\label{FRt701}
\|\tau(t)\|_{L^{\infty}}+\int_{0}^{T}{\|\nabla u(s)
\|_{L^{\infty}}\,ds}\leq C<\infty
\end{eqnarray}
for any $t\in[0, T].$
\end{lemma}
\begin{proof}[\textbf{Proof of  Lemma \ref{FRL33}}]
By the relation $G=\omega-\mathcal {R}_{\gamma}\tau$, one directly has
\begin{eqnarray}\label{FRt702}
\|\Lambda^{2\gamma-2}\omega\|_{L^{\frac{2}{\gamma-1}}}&\lesssim& \|\Lambda^{2\gamma-2} G\|_{L^{\frac{2}{\gamma-1}}}+\|\Lambda^{2\gamma-2}\mathcal
{R}_{\gamma}\tau\|_{L^{\frac{2}{\gamma-1}}}\nonumber\\&\lesssim& \|\Lambda^{\gamma} G\|_{L^{2}}+\|\tau\|_{L^{\frac{2}{\gamma-1}}}.
\end{eqnarray}
By the estimate (\ref{FRt202}), it thus gives
$$\int_{0}^{T}{\|\Lambda^{2\gamma-2}\omega(t)\|_{L^{\frac{2}{\gamma-1}}}\,dt}<\infty,$$
which together with the embedding leads to
$$\int_{0}^{T}{\|\nabla u(t)\|_{L^{\infty}}\,dt}\lesssim\int_{0}^{T}{(\|u(t)\|_{L^{2}}+\|\Lambda^{2\gamma-2}
\omega(t)\|_{L^{\frac{2}{\gamma-1}}})\,dt}<\infty.$$
On the other hand, we have
\begin{eqnarray}\frac{d}{dt}\|\tau\|_{L^{r}}\leq
 C\|\mathcal {D}u\|_{L^{r}} \|\tau\|_{L^{r}}\leq
 C\|\nabla u\|_{L^{r}} \|\tau\|_{L^{r}}.\nonumber
\end{eqnarray}
Letting $r\rightarrow\infty$, we get
\begin{eqnarray}\frac{d}{dt}\|\tau\|_{L^{\infty}}\leq
 C\|\nabla u\|_{L^{\infty}} \|\tau\|_{L^{\infty}}.\nonumber
\end{eqnarray}
The classical Gronwall inequality entails
$$\|\tau\|_{L^{\infty}}<\infty.$$
We thus complete the proof of  Lemma \ref{FRL33}.
\end{proof}

We are now ready to
prove the following key lemma.
\begin{lemma}\label{FRL34} Under the assumptions of Lemma \ref{FRL32}, there exist some constants $C$
such that for any $T>0$
\begin{eqnarray}\label{FRt215}
\|\omega(t)\|_{L^{2}}^{2}+\|\nabla\tau(t)\|_{L^{2}}^{2}+\int_{0}^{T}{\|\Lambda^{\gamma}
\omega(s)\|_{L^{2}}^{2}\,ds}\leq
C<\infty
\end{eqnarray}
for any $t\in[0, T].$
\end{lemma}
\begin{proof}[\textbf{Proof of  Lemma \ref{FRL34}}]
Just as in proving Lemma \ref{L33}, we immediately obtain
\begin{eqnarray}\label{FRt218}
&&\frac{1}{2}\frac{d}{dt}(\|\omega\|_{L^{2}}^{2}+\|\nabla
\tau\|_{L^{2}}^{2})+\|\Lambda^{\gamma} \omega\|_{L^{2}}^{2} \nonumber\\
&=&\int_{\mathbb{R}^{2}}{({\rm curl \,div}
(\tau))\omega\,dx}+\int_{\mathbb{R}^{2}}{\nabla\mathcal {D}u\cdot \nabla
\tau\,dx}-\int_{\mathbb{R}^{2}}{\nabla(u\cdot\nabla \tau): \nabla
\tau\,dx}\nonumber\\&&+
\int_{\mathbb{R}^{2}}{\nabla(\tau\Omega-\Omega\tau)\cdot\nabla\tau\,dx}\nonumber\\&:=&
\mathcal{K}_{1}+\mathcal{K}_{2}+\mathcal{K}_{3}+\mathcal{K}_{4}.
\end{eqnarray}
These terms on the right-hand side
of (\ref{FRt218}) can be estimated as follows. By the H$\rm\ddot{o}$lder inequality and the Gagliardo-Nirenberg inequality, it lead to
\begin{eqnarray}\label{FRt219}
\mathcal{K}_{1} &\leq& \|\nabla\tau\|_{L^{2}} \|\nabla w \|_{L^{2}}\nonumber\\
 &\leq& \|\nabla\tau\|_{L^{2}} \| w \|_{L^{2}}^{1-\frac{1}{\gamma}}\| \Lambda^{\gamma}w \|_{L^{2}}^{\frac{1}{\gamma}}\nonumber\\
&\leq& \frac{1}{4}\|\Lambda^{\gamma}\omega\|_{L^{2}}^{2}+C(\|\nabla \tau
\|_{L^{2}}^{2}+ \| w \|_{L^{2}}^{2}),
\end{eqnarray}
\begin{eqnarray}\label{FRt220}
\mathcal{K}_{2} &\leq& \|\nabla\tau\|_{L^{2}} \|\nabla \mathcal {D}u \|_{L^{2}}\nonumber\\
&\lesssim&\|\nabla\tau\|_{L^{2}} \|\nabla \omega \|_{L^{2}}
\nonumber\\&\leq& \frac{1}{4}\|\Lambda^{\gamma}\omega\|_{L^{2}}^{2}+C(\|\nabla \tau
\|_{L^{2}}^{2}+ \| w \|_{L^{2}}^{2}),
\end{eqnarray}
\begin{eqnarray}\label{FRt221}
\mathcal{K}_{3} &\leq& C\|\nabla u\|_{L^{\infty}}\|\nabla \tau\|_{L^{2}}^{2},
\end{eqnarray}
and
\begin{eqnarray}\label{FRt222}
\mathcal{K}_{4} &\leq& C \int_{\mathbb{R}^{2}}{|\nabla\tau|\,|\Omega|\,\nabla\tau|\,dx}+C \int_{\mathbb{R}^{2}}{|\tau|\,|\nabla\Omega|\,\nabla\tau|\,dx}
\nonumber\\&\leq&
C\|\nabla u\|_{L^{\infty}}\|\nabla \tau\|_{L^{2}}^{2}+C\|\tau\|_{L^{\infty}}
\|\nabla\Omega\|_{L^{2}}
\|\nabla\tau\|_{L^{2}}
\nonumber\\&\leq&
C\|\nabla u\|_{L^{\infty}}\|\nabla \tau\|_{L^{2}}^{2}+C\|\tau\|_{L^{\infty}}
\|\nabla \omega\|_{L^{2}}
\|\nabla\tau\|_{L^{2}}\nonumber\\&\leq&
C\|\nabla u\|_{L^{\infty}}\|\nabla \tau\|_{L^{2}}^{2}+C\|\tau\|_{L^{\infty}}
\| w \|_{L^{2}}^{1-\frac{1}{\gamma}}\| \Lambda^{\gamma}w \|_{L^{2}}^{\frac{1}{\gamma}}
\|\nabla\tau\|_{L^{2}}
\nonumber\\&\leq&
\frac{1}{4}\|\Lambda^{\gamma}\omega\|_{L^{2}}^{2}+ C\|\nabla u\|_{L^{\infty}}\|\nabla \tau\|_{L^{2}}^{2}+C\|\tau\|_{L^{\infty}}^{\frac{2\gamma}{2\gamma-1}}(\|\nabla \tau
\|_{L^{2}}^{2}+ \| w \|_{L^{2}}^{2}).
\end{eqnarray}
Inserting the four previous estimates (\ref{FRt219}), (\ref{FRt220}), (\ref{FRt221}) and (\ref{FRt222}) into (\ref{FRt218}) yields
\begin{eqnarray}\label{FRt223}
\frac{d}{dt}(\|\omega\|_{L^{2}}^{2}+\|\nabla
\tau\|_{L^{2}}^{2})+\|\Lambda^{\gamma} \omega\|_{L^{2}}^{2}\leq C(1+\|\nabla u\|_{L^{\infty}}+\|\tau\|_{L^{\infty}}^{\frac{2\gamma}{2\gamma-1}})(\|\nabla \tau
\|_{L^{2}}^{2}+ \| w \|_{L^{2}}^{2}).\nonumber
\end{eqnarray}
Noticing the estimate (\ref{FRt701}), we can conclude the desired result by applying Gronwall inequality.
Consequently, this concludes the proof of Lemma \ref{FRL34}.
\end{proof}

\begin{proof}[\textbf{Proof of Theorem \ref{Th2}}]
Applying $\Delta$ to system (\ref{FROLD}) and
taking the $L^{2}$ inner product with $\Delta u$ and
$\Delta\tau$ respectively, adding them up, we have
\begin{eqnarray}\label{FRt224}
&&\frac{1}{2}\frac{d}{dt}(\|\Delta u(t)\|_{L^{2}}^{2}+
\|\Delta\tau(t)\|_{L^{2}}^{2})+\|\Lambda^{2+\gamma}
u\|_{L^{2}}^{2} \nonumber\\
&\lesssim&\|\Delta\tau\|_{L^{2}} \|\Lambda^{3} u\|_{L^{2}}
+\Big|\int_{\mathbb{R}^{2}}{[\Delta, u\cdot\nabla]u\cdot
\Delta u\,dx}\Big|+\Big|\int_{\mathbb{R}^{2}}{ [\Delta,
u\cdot\nabla]\tau:
\Delta\tau \,dx}\Big|
\nonumber\\&&+\int_{\mathbb{R}^{2}}{\Delta(\tau\Omega-\Omega\tau):\Delta\tau\,dx}
\nonumber\\
&:=&\mathcal{ L}_{1}+\mathcal{ L}_{2}+\mathcal{ L}_{3}+\mathcal{ L}_{4}.
\end{eqnarray}
The four terms can be bounded by
\begin{eqnarray}
\mathcal{ L}_{1}&\leq&\|\Delta\tau\|_{L^{2}} \|\Delta u\|_{L^{2}}^{\frac{\gamma-1}{\gamma}}
\|\Lambda^{2+\gamma} u\|_{L^{2}}^{\frac{1}{\gamma}} \nonumber\\
&\lesssim& \frac{1}{4}\|\Lambda^{2+\gamma}
u\|_{L^{2}}^{2}+C(\|\Delta u\|_{L^{2}}^{2}+\|\Delta\tau\|_{L^{2}}^{2}),\nonumber
\end{eqnarray}
\begin{eqnarray}
\mathcal{ L}_{2}&\leq& C \|[\Delta,
u\cdot\nabla]u\|_{L^{2}}\|\Delta u\|_{L^{2}}\nonumber\\ &\leq& C
\|\nabla u\|_{L^{\infty}}\|\Delta u\|_{L^{2}}^{2},\nonumber
\end{eqnarray}
\begin{eqnarray}
\mathcal{ L}_{3}&=& \Big|\int_{\mathbb{R}^{2}}{ [\Delta,
u\cdot\nabla]\tau:
\Delta\tau \,dx}\Big|
\nonumber\\&=& \Big|\int_{\mathbb{R}^{2}}{ \Delta
u\nabla\tau
\Delta\tau \,dx}\Big|+\Big|\int_{\mathbb{R}^{2}}{ \nabla
u \nabla\nabla \tau
\Delta\tau \,dx}\Big|
\nonumber\\
&\leq& C\|\Delta u\|_{L^{\infty}}\|\nabla\tau\|_{L^{2}}\|\Delta\tau\|_{L^{2}}+C\|\nabla
u\|_{L^{\infty}}\|\Delta\tau\|_{L^{2}}^{2}\nonumber\\
&\leq& C\|\Delta u\|_{L^{2}}^{\frac{\gamma-1}{\gamma}}\|\Lambda^{2+\gamma}u\|_{L^{2}}^{\frac{1}
{\gamma}}\|\nabla\tau\|_{L^{2}}\|\Delta\tau\|_{L^{2}}+C\|\nabla
u\|_{L^{\infty}}\|\Delta\tau\|_{L^{2}}^{2}
\nonumber\\
&\leq&
\frac{1}{4}\|\Lambda^{2+\gamma}
u\|_{L^{2}}^{2}+C\|\nabla\tau\|_{L^{2}}^{\frac{2\gamma}{2\gamma-1}}(\|\Delta u\|_{L^{2}}^{2}+
\|\Delta\tau\|_{L^{2}}^{2})+C\|\nabla
u\|_{L^{\infty}}\|\Delta\tau\|_{L^{2}}^{2},\nonumber
\end{eqnarray}
and
\begin{eqnarray}
\mathcal{L}_{4} &\leq& C \int_{\mathbb{R}^{2}}{\Delta(\tau\Omega)\Delta\tau\,dx}
\nonumber\\&\leq&
C\|\nabla u\|_{L^{\infty}}\|\Delta\tau\|_{L^{2}}^{2}+C\|\tau\|_{L^{\infty}}
\|\Delta\Omega\|_{L^{2}}
\|\Delta\tau\|_{L^{2}}
\nonumber\\&\leq&
C\|\nabla u\|_{L^{\infty}}\|\Delta\tau\|_{L^{2}}^{2}+C\|\tau\|_{L^{\infty}}
\|\Delta u\|_{L^{2}}^{\frac{\gamma-1}{\gamma}}\|\Lambda^{2+\gamma}
u\|_{L^{2}}^{\frac{1}{\gamma}}
\|\Delta\tau\|_{L^{2}}
\nonumber\\
&\leq&
\frac{1}{4}\|\Lambda^{2+\gamma}
u\|_{L^{2}}^{2}+C\|\tau\|_{L^{\infty}}^{\frac{2\gamma}{2\gamma-1}}(\|\Delta u\|_{L^{2}}^{2}+
\|\Delta\tau\|_{L^{2}}^{2})+C\|\nabla
u\|_{L^{\infty}}\|\Delta\tau\|_{L^{2}}^{2}.\nonumber
\end{eqnarray}
Substituting all the preceding estimates into (\ref{FRt224}), we find that
\begin{eqnarray}\label{FRt231}
&&\frac{d}{dt}(\|\Delta u(t)\|_{L^{2}}^{2}+
\|\Delta\tau(t)\|_{L^{2}}^{2})+\|\Lambda^{2+\gamma}
u\|_{L^{2}}^{2}+\|\Lambda^{2+\alpha}
\tau\|_{L^{2}}^{2}\nonumber\\
&\lesssim&C(1+\|\nabla
u\|_{L^{\infty}}+\|\nabla\tau\|_{L^{2}}^{\frac{2\gamma}{2\gamma-1}}+
\|\tau\|_{L^{\infty}}^{\frac{2\gamma}{2\gamma-1}})(\|\Delta u\|_{L^{2}}^{2}+
\|\Delta\tau\|_{L^{2}}^{2}).\nonumber
\end{eqnarray}
We thus get by the Gronwall inequality that
\begin{eqnarray}\label{FRt231}\|\Delta u(t)\|_{L^{2}}^{2}+
\|\Delta\tau(t)\|_{L^{2}}^{2}+\int_{0}^{T}{(\|\Lambda^{2+\gamma}
u\|_{L^{2}}^{2}+\|\Lambda^{2+\alpha}
\tau\|_{L^{2}}^{2})(t)\,dt}<\infty.
\end{eqnarray}
Finally, the $H^{s}$-estimate with $s>2$ can be easily deduced with the help of the above estimate (\ref{FRt231}).
Thus, we conclude the proof of Theorem \ref{Th2}.
\end{proof}

\end{document}